\documentclass[11pt, reqno]{amsart}

\usepackage[utf8]{inputenc}
\usepackage[T1]{fontenc}
\usepackage{lmodern}
\usepackage{amsmath, amssymb, amsthm}
\usepackage{mathabx}
\usepackage[colorlinks=true, linkcolor=black, citecolor=blue, urlcolor=blue]{hyperref}
\usepackage{todonotes}

\usepackage{geometry}
\usepackage{mathrsfs}
\usepackage{enumerate}
\usepackage{cite}
\usepackage{algorithm}
\usepackage{cleveref}
\usepackage{graphicx}
\usepackage{float} % für [H]
\usepackage[noend]{algpseudocode}
\usepackage{tikz}
\usetikzlibrary{arrows.meta, positioning, fit,backgrounds}

\newtheorem{theorem}{Theorem}[section]
\theoremstyle{definition}
\newtheorem{lemma}[theorem]{Lemma}
\theoremstyle{definition}
\newtheorem{corollary}[theorem]{Corollary}
\theoremstyle{definition}
\newtheorem{proposition}[theorem]{Proposition}
\newtheorem{assumption}{A}[section]
\theoremstyle{definition}
\newtheorem{definition}[theorem]{Definition}
\newtheorem{example}[theorem]{Example}
\theoremstyle{remark}
\newtheorem{remark}[theorem]{Remark}

\title[Optimization on Weak Riemannian Manifolds]{Optimization on Weak Riemannian Manifolds}

\newcommand{\R}{\mathbb{R}}
\newcommand{\SSS}{\mathbb{S}}
\newcommand{\ev}{\mathrm{ev}}
\newcommand{\Hess}{\mathrm{Hess}}
\DeclareMathOperator{\grad}{grad}
\newcommand{\Dt}{\frac{\mathrm{D}}{\mathrm{d}t}}
\newcommand{\dt}{\frac{\mathrm{d}}{\mathrm{d}t}}

\author{Valentina Zalbertus}
\address{Georg-August-University Göttingen, Institute for Applied and Numerical Mathematics, Lotzestr.~16-18, 37083 Göttingen}
\email{v.zalbertus@stud.uni-goettingen.de}

\author{Max Pfeffer}
\address{Georg-August-University Göttingen, Institute for Applied and Numerical Mathematics, Lotzestr.~16-18, 37083 Göttingen}
\email{m.pfeffer@math.uni-goettingen.de}

\author{Alexander Schmeding}
\address{Norwegian University of Science and Technology,
Department of Mathematical Sciences,
Alfred Getz' vei 1, Trondheim}
\email{alexander.schmeding@ntnu.no}

\date{\today}

\begin{document}

\begin{abstract}
Riemannian structures on infinite-dimensional manifolds arise naturally in shape analysis and shape optimization. These applications lead to optimization problems on manifolds which are not modeled on Banach spaces. The present article develops the basic framework for optimization via gradient descent on weak Riemannian manifolds leading to the notion of a Hesse manifold. Further, foundational properties for optimization are established for several classes of weak Riemannian manifolds connected to shape analysis and shape optimization.  
\end{abstract}

\maketitle

\tableofcontents

\section{Introduction}

In recent years, standard first and second order methods from continuous optimization in euclidean space have been generalized to Riemannian manifolds, thus kick-starting the very active field of Riemannian optimization. In particular, much research has been done for matrix manifolds \cite{absil2008optimization,Bou23}. Even nonsmooth optimization on smooth Riemannian manifolds has been studied extensively \cite{Chen2020,Si2024}. In higher dimensions, it has been recognized that tensor trees form Riemannian manifolds, allowing for the adaptation of methods on matrix manifolds \cite{Kressner2014}. 

However, things are much less clear when it comes to Riemannian manifolds of {\em infinite} dimension. 
%These are observed, for example, in shape optimization \cite{???} or ???. 
For the special case of Hilbert manifolds, optimization using gradient descent is classical, see, e.g., the literature overview in \cite{SaZ04}. There are natural geometric applications for gradients and their flows on Hilbert (sub-)manifolds: morse theory \cite{Pal63,PaS64}, energy functions \cite{DaF25} for knot deformations, optimal transport on Wasserstein space (see e.g. \cite{Otto01,Lott07,NEURIPS2021} for discussions). Beyond Hilbert manifolds gradient descent techniques typically use conjugate gradients in (reflexive) Banach spaces, see e.g. \cite{Fee20,FaM20,T25}.

In the present article we discuss basic theory for optimization on infinite-dimensional manifolds using gradients and Hessians beyond the setting of Hilbert manifolds. One of several challenges arising in the passage to infinite-dimensions is the split between different regimes of Riemannian geometry: Hilbert manifolds admit {\em strong} Riemannian metrics but manifolds modeled on more general spaces only admit {\em weak} Riemannian metrics, see \cite{Sch23}.

For the strong Riemannian metrics, the theory develops along the finite-dimensional lines, see e.g. \cite{Pal63,Eli74,Kli95}. Since infinite-dimensional manifolds are not locally compact, extra conditions (e.g. Palais-Smale condition (C), \cite{PaS64}) are required to ensure convergence of the gradient sequences. Second order theory using the Riemannian Hessian gets more involved on Hilbert manifolds.

Beyond Hilbert manifolds, every Riemannian metric is necessarily a weak Riemannian metric, i.e., the induced inner products on the tangent spaces are only continuous and do not induce the native topology. Even on an open subset of an infinite-dimensional Hilbert space, the inner product induced by a weak Riemannian metric is in general not equivalent to the Hilbert space product of the model space. Weak Riemannian metrics arise in many applications. We list several settings where gradients, gradient flows and questions from optimization are of central interest in an infinite-dimensional setting:
\begin{itemize}
\item As pioneered by V.I. Arnold, certain partial differential equations (PDEs) lift to geodesic equations on manifolds of Sobolev mappings (cf. \cite[Chapter 7]{Sch23}). These are Hilbert manifolds with weak Riemannian metrics, cf. e.g. \cite{Otto01}.
\item Shape analysis studies invariant metrics and flows on weak Riemannian manifolds of mappings and diffeomorphism groups, cf. e.g.  \cite{MM07,Mic15,BBM14,MMM13,You19}. Here optimization is relevant in large deformation diffeomorphic metric matching (LDDMM), \cite{Tro98}. See e.g. \cite{BKK23} for a concrete example involving the gradient flow.
\item Shape optimization studies gradients for weak Riemannian metrics on infinite-dimensional manifolds, see e.g. \cite{LaW25,LaPaW25}.
\item (time-)evolving embedded manifolds and evolution equations on them lead to gradient flows on weak Riemannian manifolds. The curve-shortening flow studied by Gage and Hamilton and related flows are of this type, cf. \cite{GaH86,SWW23}
\end{itemize}
The state of the art to treat these problems is to employ one of the following strategies:
Treat the qualitative behavior of gradient flows. For example for time-evolving and shape manifolds, development of singularities of the gradient flows and geodesic equations are studied without directly employing numerical methods, \cite{GaH86,Otto01}. 

For optimization schemes base on infinite-dimensional manifolds, there are two main approaches: In many relevant examples, the infinite-dimensional gradient equations can be translated to finite-dimensional (partial) differential equations. These are then numerically solved using PDE methods (e.g. \cite{BaCaKaKaNaP24,LaW25,LaPaW25}). In the Hilbert manifold setting, discretization of the equations are applied together with conditions assuring convergence and convergence rate, see e.g. \cite{Eli74,Pal63,SaZ04}. These techniques have been generalized to Banach manifolds (e.g. \cite{FaM20,Fee20,Tro98,DaF25,T25}) using weaker notions of gradients and dualities not necessarily induced by (weak) Riemannian metrics. 
These approaches either require strong settings (strong metrics, Hilbert manifolds) or exploit connections to finite dimensional geometry for the discretization and computation of the descent scheme. To the best of our knowledge, a general investigation of basic optimization algorithms, and in particular discretization, for weak Riemannian manifolds is so far missing.

One aim of the present article is to provide an introduction to basic optimization techniques on infinite-dimensional manifolds in the weak setting. We highlight pitfalls and challenges arising on Riemannian manifolds beyond the Hilbert setting. 
Further, fundamental optimality conditions and convergence results for optimization on weak Riemannian manifolds are provided. We will not tackle the more involved question of discretization of infinite-dimensional manifolds. This is postponed to later work and we demonstrate the algorithms using symbolic computations which avoid discretization. While much of the classical intuition from finite-dimensional optimization (as presented in \cite{Bou23}) carries over, the absence of the Hilbert/Banach space structure makes it a priori unclear, in which sense standard optimality conditions generalize to weak Riemannian manifolds.

\begin{theorem}[First-Order Optimality]
    Let $f\colon M\to \R$ be continuously differentiable on a weak Riemannian manifold $M$. Then every local minimizer $p\in M$ satisfies
    \begin{equation*}
        \nabla f(p) = 0,
    \end{equation*}
    where $\nabla f$ denotes Riemannian gradient of $f$.
\end{theorem}
This recovers the familiar necessary condition from finite-dimensional optimization~\cite{Bou23}.
To extend first-order optimality conditions to algorithms, we show that under an additional assumption ensuring sufficient structure on weak Riemannian manifolds, the classical finite-dimensional convergence result for the Riemannian gradient descent \cite{Bou23} carries over to our present setting.
\begin{theorem}
    All accumulation points of the sequence of iterates $(p_n)_{n\in \mathbb{N}}$ generated by the Riemannian descent algorithm are critical points, and 
    \begin{equation*}
        \lim \limits_{n\to \infty}\vvvert \nabla f(p_n)\vvvert = 0.
    \end{equation*}
where $\vvvert \cdot \vvvert$ is the norm induced by the weak Riemannian metric.
\end{theorem}
Second order optimality is more complicated due to intricate structure arising in the critical points of the Hessian (cf. e.g.\ Example \ref{ex:Eli_curves}). Nevertheless, one can prove the following: 
\begin{theorem}[Second-Order Optimality]
    A point $p\in M$ with $\nabla f(p) = 0$ and $\Hess f(p)$ positive-definite is a local minimizer if and only if the Riemannian Hessian is coercive at that point, i.e. there exists $\mu >0$ such that
    \begin{equation*}
        g_p(\Hess f(p)[v],v) \geq \mu \vvvert v\vvvert_p, \quad \forall v\in T_pM.
    \end{equation*}
\end{theorem}
Unlike the finite-dimensional setting--where positive definiteness of the Hessian suffices--coercivity is more restrictive here, failing to follow from positive definiteness on weak Riemannian manifolds. 
Note that this describes a typical phenomenon beyond Hilbert spaces. For example it is well known that convexity properties on functions used in finite dimensional optimization, typically force a Banach space to either be reflexive or even a Hilbert space (see e.g. \cite{BaW25,3279480}).

To establish second-order optimality conditions that provide, in addition to necessary conditions, a sufficient condition for local minima, we require several additional properties of the underlying weak Riemannian manifold. These properties ensure that the Hessian is well behaved and allow us to draw conclusions about local extrema. A weak Riemannian manifold satisfying these properties will be called a \textit{Hesse manifold}. We show that Hesse manifolds constitute a refinement of the existing classification into weak, robust, and strong Riemannian manifolds. In particular, we demonstrate  that:
\begin{theorem}
    Every robust Riemannian $C^\infty$- manifold $(M,g)$ is a Hesse manifold. 
\end{theorem}
We then study the robust metrics introduced in \cite{MMM13} with respect to their application in optimization. As a new result, we prove that the class of elastic metrics from shape analysis are robust.
Summing up, this leads to the following hierarchy of Riemannian manifolds:
\begin{center}
\begin{tikzpicture}[
    node distance=0.5cm,
    box/.style={
        rectangle,
        draw, fill=white,
        rounded corners,
        minimum width=2.8cm,
        minimum height=1cm,
        align=center
    },
    box2/.style={
        rectangle,
        draw, fill=white,
        rounded corners,
        minimum width=4cm,
        minimum height=1cm,
        align=center
    },
    dashedgroup/.style={
        rectangle,
        draw,
        dashed,
        rounded corners,
        inner sep=0.2cm
    },
    examplegroup/.style={
        rectangle,
        fill=gray!10,
        rounded corners,
        inner sep=0.2cm
    },
    solidgroup/.style={
        rectangle,
        draw,
        rounded corners,
        inner sep=0.25cm
    },
        widebox/.style={
        rectangle,
        draw, fill=white,
        rounded corners,
        minimum width=6.1cm, % spans two boxes + spacing
        minimum height=1cm,
        align=center
    },
    arrow/.style={
        ->,
        thick
    }
]

% Middle row
\node[box2] (m2) {Strong\\Riemannian};
\node[box, right=of m2] (m3) {Robust\\Riemannian};
\node[box, right=of m3] (m4) {Hesse\\manifold};
\node[box, right=of m4] (m5) {Weak\\Riemannian};

% Top node
\node[box2, above=of m2] (m1) {dim $< \infty$ \\paracompact};

% Bottom row
\node[box2, below=of m2] (e1){Grossmann's ellipsoid\\ Example \ref{ex:GM_ellipsoid}};
\node[box, below=of m3] (d1) {};
\node[widebox, right=of e1] (e23) {$L^2$-metric; elastic metric\\ Proposition \ref{prop:L2_robust}; Remark \ref{rem:elastic}};
\node[box, below=of m5] (e4) {twisted $\ell^2$\\ Example \ref{example: Metric Spray counterexp}};

% Horizontal arrows
\draw[arrow] (m2) -- (m3);
\draw[arrow] (m3) -- (m4);
\draw[arrow] (m4) -- (m5);

% Vertical arrows
\draw[arrow] (m1) -- (m2);
\draw[arrow] (m2) -- (e1);
\draw[arrow] (m3) -- (d1);
\draw[arrow] (m5) -- (e4);

% Inner dashed group

% Outer solid group (Hilbert manifold)
\node[solidgroup, fit=(m1)(m2)(e1),
      label={[rotate=90, anchor=center, fill=white, inner sep=2pt]
             west:\textbf{Hilbert Manifold}}] {};

% Examples 
\begin{pgfonlayer}{background}
\node[dashedgroup, fit=(m2)(m3)(m4)(m5)(e1)(e23)(e4),
      label={[anchor=west, fill=white, xshift=6cm, yshift=0.15cm]north west: \textbf{(possibly) Infinite-Dimensional manifold}}] (g1) {};
\node[examplegroup, fit=(e1)(e23)(e4),
      label={[anchor=center,fill=white, inner sep=2pt] below:\textbf{Examples}}] {};
\end{pgfonlayer}
\end{tikzpicture}

\end{center}

The structure of the article is as follows:  To establish Riemannian optimization on weak Riemannian manifolds, we first address the primary structural challenges. Section~\ref{sec: weak Riem mfds} introduces two fundamental restrictions enabling Riemannian optimization in this generality, presents examples of pathological behavior without them, and verifies that these restrictions preserve the essential structure of weak Riemannian manifolds.

Building on this foundation, Section~\ref{sec: Optimality conditions} derives first- and second-order optimality conditions in terms of the Riemannian gradient and Hessian. Section~\ref{sec: RGD} introduces the Riemannian gradient descent method and analyzes its convergence, showing that classical results carry over under mild additional conditions.

We then introduce two key classes - strong and robust Riemannian manifolds - focusing on the latter's construction and structural properties (Section~\ref{sec: Hesse mfds}), while proving simplifications for the former. Finally, Section~\ref{sec: computation} provides explicit formulas for Riemannian gradients and Hessians, complemented by numerical examples (Section ~\ref{sec: num exp}).\smallskip

\textbf{Acknowledgements} V.Z. was funded by the German research foundation (DFG – Projektnummer 448293816), she thanks the Department of mathematics at NTNU for the hospitality during a research stay while part of this work was conducted.

\section{Preliminaries}
Weak Riemannian manifolds are often modeled on locally convex spaces which are in general not Banach manifolds. 
The usual calculus, also called Fréchet differentiability, has to be replaced. We employ Bastiani calculus, see \cite[Section 1.4]{Sch23}, which is based on directional derivatives. This means that a continuous function $f\colon E \supseteq U \rightarrow F$ on an open subset of a locally convex space is $C^1$ if for every $x \in U, v\in E$ the directional derivative 
$$df(x;v) := \lim_{h\rightarrow 0} h^{-1}(f(x+hv)-f(x))$$
exists and yields a continuous map $df\colon U\times E \rightarrow F$. Using iterated directional derivatives, one likewise defines $C^k$-mappings for $k\in \mathbb{N}$. A map which is $C^k$ for all $k\in \mathbb{N}$ is called smooth or $C^\infty$. The usual assertions such as linearity of the derivative and the chain rule remain valid. 

As the chain rule is valid, we can define as in finite dimensions, manifolds via charts. A manifold is called a \emph{Hilbert/Banach/Fr\'{e}chet-manifold} if all the modeling spaces of the manifold are Hilbert/Banach/Fr\'{e}chet spaces. Further, for a manifold $M$ the tangent spaces $T_pM$ are defined via equivalence classes of curves \cite[Def.~1.41]{Sch23} and are canonically isomorphic to the model space of the manifold. Similarly the tangent bundle and differentiability of mappings on manifolds can be defined. For the tangent map of a $C^1$-map $f\colon M \rightarrow N$ we will write
\[
D_pf\colon T_pM\to T_{f(p)}N, \quad [\gamma]\mapsto [f\circ\gamma].
\]
For a vector bundle $\pi \colon E \rightarrow M$ on a smooth manifold, we will write $\Gamma(E)$ for the space of smooth bundle sections. In the special case that $E=TM$ is the tangent bundle, we also write $\mathcal{V}(M):=\Gamma(TM)$.

When establishing Riemannian metrics on locally convex manifolds beyond the Hilbert setting, a crucial distinction arises between \emph{weak} and \emph{strong} Riemannian metrics, essential for the subsequent optimization.
\begin{definition}[Weak/Strong Riemannian Manifold]
  Let $M$ be a $C^1$-manifold. A \emph{weak Riemannian metric} $g$ on $M$ is a smooth map on the Whitney-sum
  \[
  g \colon TM \oplus TM \to \mathbb{R}, \quad (v_p, w_p) \mapsto g_p(v_p, w_p),
  \]
  such that $g_p$ is symmetric, bilinear on $T_p M \times T_p M$, and $g_p(v,v) \geq 0$ with equality iff $v = 0$.
  If the topology on $(T_p M, g_p)$ coincides with the subspace topology of $T_p M \subset TM$, then $g$ is \emph{strong}.
  We then call $(M, g)$ a \emph{weak/strong Riemannian manifold}.
\end{definition}

\begin{example}\label{ex:basic}
Let $(H, \langle \cdot,\cdot \rangle)$ be a Hilbert space and $K$ a compact manifold with volume form $\mu$. Recall from \cite[Chapter 2]{Sch23} that the space of smooth functions $C^\infty (K,H)$ is a Frech\`{e}t space and if $\Omega$ is an open subset of $H$, then $C^\infty (K,\Omega)$ is open (whence a submanifold of $C^\infty (K,H)$). Integration against the volume form yields metrics on $C^\infty(K,\Omega)$, namely the $L^2$-metric
\begin{align*}
g^{L^2} \colon C^\infty (K,\Omega) \times C^\infty (K,H)^2 &\rightarrow \R ,\quad  (f,(v,w))\mapsto g^{L^2}_f (v,w):= \int_K \langle v(x) , w(x)\rangle \, d \mu(x)
\end{align*}
(where we identify the tangent bundle using standard arguments as in \cite[Chapter 2]{Sch23}). Further, for the case $K=\mathbb{S}^1$, one also considers a reparametrization invariant version of the $L^2$-metric
\begin{align*}
g_{\text{inv}} \colon C^\infty (\mathbb{S}^1,\Omega) \times C^\infty (\mathbb{S}^1,H)^2 &\rightarrow \R , (c,(v,w)) \mapsto g_{\text{inv},c} (v,w)= \int_K \langle v(x),w(x)\lVert  \dot{c}(x) \rVert d\mu(x) 
 \end{align*}
It is not hard to see (cf. \cite[Chapter 4 and 5]{Sch23}) that these mappings yield weak Riemannian metrics. Since strong Riemannian metrics can only exist on Hilbert manifolds (cf. Lemma \ref{definition:alternative strong Riemannian manifold}), $C^\infty (K,\Omega)$ with the above metrics can not be a strong Riemannian manifold.
\end{example}

Since we operate beyond the Banach setting, there is no natural norm on the spaces we consider. Although the inner products induce norms, these do not generate the natural topology, and in particular, the spaces are not complete with respect to these norms. 
\begin{remark}
To avoid confusion, we write $\vvvert v \vvvert _p := \sqrt{g_p(v,v)}$ for the norm on $T_pM$ induced by the inner product $g_p$, which need not be complete, and $\|v\|$ for a Banach norm, if we are working in the Banach case.
\end{remark}
To facilitate Riemannian optimization in our setting, we introduce:
\begin{definition}[Riemannian Gradient]\label{def:riem-grad}
  Let $(M,g)$ be a weak Riemannian $C^1$-manifold and $f \colon M \to \mathbb{R}$ a $C^1$-map. A vector field $\nabla f$ satisfying
  \[
  D_pf(v) = g_p(\nabla f(p), v) \quad \forall\, v \in T_p M
  \]
  is the \emph{Riemannian gradient} of $f$.
\end{definition}

\begin{definition}[Riemannian Hessian]\label{def:riem-hess}
  Let $(M,g)$ be a $C^2$-manifold with first-order\footnote{a connection is first order if its value at a point depends at most on the $1$-jets of the sections at the point. See Remark \ref{rem:fo-connection}. Every connection on a finite dimensional manifold is of first order.} Levi--Civita connection $\nabla$, and $f \colon M \to \mathbb{R}$ a $C^2$-function with Riemannian gradient $\nabla f$. The  \emph{Riemannian Hessian} of $f$ at $p$ is the map
  \[
  \Hess f(p) \colon T_p M \to T_p M, \quad v \mapsto \nabla_v \nabla f(p).
  \]
\end{definition}
All definitions and results from infinite-dimensional differential geometry follow \cite{Sch23}. For the readers convenience we recall some essential technical objects in Appendix \ref{App:diffgeo}.

\section{Weak Riemannian Manifolds in Optimization}
\label{sec: weak Riem mfds}
To introduce the subsequent chapters on optimization on weak Riemannian manifolds, we first specify the setting in which Riemannian optimization techniques can be applied. Although the objective of this work is to develop optimization methods on spaces as general as possible - namely weak Riemannian manifolds - the weak structure of the underlying geometry requires us to impose several structural assumptions in order to establish a well-defined framework.

Since our optimization approach relies on Riemannian methods, we focus on first- and second-order differential objects, in particular the Riemannian gradient and the Riemannian Hessian. These quantities are essential for the formulation and analysis of first- and second-order optimality conditions and gradient-based optimization algorithms.

On weak Riemannian manifolds, however, these objects are not available in general. Recall that for a weak Riemannian $C^1$-manifold $(M,g)$ the Riemannian gradient of a $C^1$-function $f$ is defined by the unique vector field satisfying $D_pf(v) = g_p(\nabla f(p),v)$ for all $v \in T_pM$. Since on weak Riemannian manifolds the musical morphism between the tangent bundle and its dual is not necessarily surjective \cite[4.4]{Sch23}, the existence of the Riemannian gradient of a function cannot be guaranteed. The following example demonstrates a situation in which the Riemannian gradient fails to exist on the tangent space under consideration. 
\begin{example}\label{ex:Sobolev}
    We consider the space $\text{Imm}(\SSS^1, \R^2)$ of all smooth immersions. It is an open subset of $C^\infty(\SSS^1, \R^2)$ (see \cite[Example 4.6]{Sch23} for details) and we use the (restriction of the) invariant metric from Example \ref{ex:basic} to define an invariant $H^1-$ metric:
    \[
g_{inv, c}^{H^1} (u,v):=g_{inv,c}(u,v) +g_{inv,c}(\dot{u},\dot{v}).
\]
In \cite[Section 4]{Sch23} it has been shown, that $\big (\text{Imm}(\SSS^1,\R^2\big),g_{inv,c}^{H^1}\big)$ is indeed a weak Riemannian manifold.
We then consider the length function
\[\mathcal{L} \colon \text{Imm}(\SSS^1,\R^2) \rightarrow \R, \mathcal{L}(c):=\int_{\SSS^1} |\dot{c}|\, d\mu.\]
In \cite[Section 4.1]{SWW23} the invariant $H^1$-gradient of the length function $\mathcal{L}$ was computed using a Green's function to solve the arising ODE. Using the arc-length reparametrization for $c$, we write $\gamma \colon [0,L]\rightarrow \R^2, s \mapsto c(\exp(\mathrm{i}s/2\pi)$ with $L:=\mathcal{L}(c)$ and the Riemannian gradient becomes: 
\begin{align}\label{eq:H1_grad_length_2}\nabla\mathcal{L} (s) = \gamma(s) + \int_{0}^{L} \gamma(t) \frac{\cosh\left(|s-t|-\frac{L}{2}\right)}{2\sinh\left(-\frac{L}{2}\right)}dt.\end{align}
Now \eqref{eq:H1_grad_length_2} will in general not be differentiable in $s$ (i.e. in the contribution by Green's function), whence the Riemannian gradient of $\mathcal{L}$ does not exist as an element in $T\text{Imm} (\SSS^1,\R^2)$ (or for that matter in the tangent space of the one time continuously differentiable immersion which is the context studied in \cite{SWW23}). Here the gradient only exists as an element in the completion of the tangent space, which can be identified with the space $H^1 (\SSS^1,\R^2)$ of all Sobolev $H^1$-functions.
\end{example}
\begin{remark}
The gradient flow induced by the length function with respect to the invariant $L^2$-metric on the immersions (instead of the $H^1$-metric we used) corresponds to the famous curve shortening flow studied in \cite{GaH86}. 
With respect to the invariant $H^1$-metric, the corresponding gradient flow has been studied in \cite{SWW23}.
\end{remark}
Nevertheless, assuming the existence of a Riemannian gradient does not turn out to be overly restrictive, since its existence does not, for instance, imply that the metric is strong. In Section ~\ref{sec: computation}, we present several examples illustrating the computation of Riemannian gradients on weak Riemannian manifolds. In particular, Example \ref{example: Riemannian Gradient on Imm} provides an explicit computation of the Riemannian gradient of the length function $\mathcal{L}$ on the space of smooth immersion $\text{Imm}(\SSS^1,\R^2)$ endowed with the invariant $L^2-$metric, thereby demonstrating that the existence of the Riemannian gradient of a function depends not only on the function itself but also on the chosen metric.

In the context of Riemannian optimization, where the structure of the Riemannian gradient is essential, but cannot be guaranteed when working on weak Riemannian manifolds, we introduce the following definition for notational convenience.
\begin{definition}
    A $C^1$- function $f\colon M \to \R$ on a weak Riemannian $C^1$- manifold $(M,g)$ is called a \textit{gradient-admitting function} (abbreviated \textit{gaf}) if the Riemannian gradient $\nabla f(p)$ exists for all $p \in M$.
\end{definition}
In addition to the Riemannian gradient, the Riemannian Hessian encodes second-order information about the local behavior of the function. Consider a weak Riemannian $C^\infty$- manifold $M$ that admits a first-order Levi-Civita connection $\nabla$. For a gradient-admitting $C^2$- function $f$ on $M$, recall, that the Riemannian Hessian of $f$ at $p\in M$ is defined by
\begin{equation*}
    \Hess f(p)[u] = \nabla _u \nabla f, \quad u\in T_pM. 
\end{equation*}
Consequently, the definition of the Riemannian Hessian requires not only the existence of the Riemannian gradient but also the availability of a first-order Levi-Civita connection. This imposes an additional structural restriction on the underlying manifold. In particular, on weak Riemannian manifolds such a connection does not exist in general. An explicit example of a weak Riemannian manifold without a Levi-Civita connection is given in \cite[p.12]{BBM14}. 

However, the existence of a Levi-Civita connection alone is still not sufficient for our subsequent analysis. In order to carry basis-independent arguments, we additionally require the existence of a metric spray, cf. Appendix \ref{App:diffgeo}. A spray is a second order vector field which, when compatible with the metric, plays the same role as the Christoffel symbols. Such a spray not only induces a first-order Levi-Civita connection, but also provides the covariant derivative structure necessary for intrinsic arguments. Similarly to the Levi-Civita connection, a metric spray does not exist on weak Riemannian manifolds in general. %The following example, taken from \cite[4.22]{Sch23}, is a weak Riemannian manifold without a metric spray.
\begin{example}
\label{example: Metric Spray counterexp}
    Consider the Hilbert space $M =\big(\ell^2,\langle \cdot ,\cdot \rangle\big)$ of all square-summable real sequences equipped with the weak Riemannian metric
    \begin{equation*}
        g\colon T\ell^2 \oplus T\ell^2 \to \R, \quad T_p\ell^2 \times T_p\ell^2 \ni \big((x_n)_n,(y_n)_n\big) \mapsto e^{-\|p\|^2}\sum_{n\in \mathbb{N}}\frac{x_ny_n}{n^3}.
    \end{equation*}
    As shown in \cite[4.22]{Sch23}, this metric does not admit a metric spray.
\end{example}
By contrast, \cite[5.7]{Sch23} computes the metric spray for a large class of weak Riemannian manifolds of the form $\big(C^\infty(\SSS^1,M),g^{L^2}\big)$, where $(M,g)$ is a strong Riemannian manifold and $g^{L^2}$ denotes the induced $L^2$ metric, showing that this additional assumption does not imply that the metric is strong.

However, Example \ref{example: Metric Spray counterexp} demonstrates that additional structural assumptions are necessary to ensure the existence and well-posedness of the Riemannian Hessian. Accordingly, the following definition establishes notation and identifies the class of weak Riemannian manifolds considered in this work.
\begin{definition}
\label{definition: Hesse manifold}
    A weak Riemannian $C^\infty$- manifold $(M,g)$ is called a \textit{Hesse manifold} if it admits a metric spray $S_g$.
\end{definition}

\section{Optimality Conditions}
\label{sec: Optimality conditions}
In this chapter, we derive first- and second-order optimality conditions for optimization on weak Riemannian manifolds under the structural assumptions introduced in the previous chapter. The goal is to show that, once these restrictions are imposed, the local optimality theory closely parallels the one on strong- or finite-dimensional Riemannian manifolds. 

Our exposition follows the framework developed by Boumal in \cite{Bou23} for finite-dimensional Riemannian manifolds. We adopt his definition of critical points, Riemannian gradients and Riemannian Hessians, and adapt the corresponding arguments to the present setting of weak Riemannian manifolds. In particular, we show that under the stated assumptions, first-order necessary optimality conditions can be formulated in terms of vanishing Riemannian gradients. While second-order conditions in the finite-dimensional setting typically only require positive definiteness of the Riemannian Hessian to guarantee a local minimum, in the infinite-dimensional setting considered here positive definiteness alone is not sufficient. Instead, an additional requirement is needed: the Riemannian Hessian must be coercive at the point of interest. These results justify the use of classical optimization intuition in the more general weak Riemannian setting for first-order conditions; however, this intuition does not carry over to second-order conditions, where additional assumptions and analytical tools are required to rigorously establish local optimality.

Throughout this chapter, $(M,g)$ denotes a weak Riemannian $C^1$-manifold.
\subsection{First-Order Optimality Conditions}
As a first step towards establishing optimization conditions on weak Riemannian manifolds, we consider the notion of critical points. In the finite-dimensional and strong Riemannian setting, critical points are characterized by the vanishing of the Riemannian gradient and are directly linked to first-order necessary conditions.

In the present weak Riemannian setting, however, this characterization is not immediate, as the definition of differentials and tangent spaces relies on Bastiani calculus rather than on a Hilbert space structure. We therefore begin by verifying that Boumal's definition of critical points is compatible with the differential structure adopted here.  
\begin{definition}
    Let $f\colon M \to \mathbb{R}$ be a $C^1$-map. A point $p \in M$ is called a \textit{critical point} of $f$, if $(f \circ \gamma)'(0) \geq 0$ for all $C^{1}$-curves $\gamma$ on $M$ passing through $p$.
\end{definition}
Despite the weak Riemannian structure, critical points admit the same characterization as in the finite-dimensional setting: critical points can be characterized equivalently by the vanishing of the differential and by the vanishing of the Riemannian gradient. The calculations are the same as in the finite dimensional setting and, for the readers convenience, we highlight only where the weak structure is needed. 
\begin{proposition}
\label{prop: critical point derivative}
    Let $f\colon M \to \mathbb{R}$ be $C^1$ and $p \in M$. The point $p$ is a critical point of $f$ if and only if 
    \begin{enumerate}
    \item $D_pf(v) = 0$ for all $v \in T_pM$,
    \item $\nabla f(p)=0$ if $f$ is a gaf.
    \end{enumerate}
  Finally, every local minimizer of $f$ is a critical point.  
\end{proposition}
\begin{proof}
The equivalence to (1) and the addendum can be proved exactly as in the finite dimensional case. See e.g. \cite[Proposition 4.5.]{Bou23} which only uses the continuity of $f\circ c$ for a smooth curve $c$ on $M$. For (2) we observe that as 
\begin{align}
D_pf(v)=g_p(\nabla f (p),v)=0, \forall v \in T_pM,\label{eq:nondeg}\end{align}
we see that (1) implies (2) as a weak Riemannian metric is nondegenerate and thus \eqref{eq:nondeg} implies that the gradient vanishes if and only if $p$ is critical. 
\end{proof}

This result enables us to establish the fundamental link between minimizers and critical points. Consequently, the classical first-order necessary optimality condition remains valid in the weak Riemannian framework considered here. This provides the foundation for the second-order analysis developed below.

\subsection{Second-Order Optimality Conditions}
We now establish sufficient second-order optimality conditions on Hesse manifolds, that is, manifolds equipped with a Levi-Civita connection induced by a metric spray. The metric spray framework allows us to define covariant derivatives of vector fields along curves in a basis-independent manner. This intrinsic notion of differentiation is crucial for formulating a second-order Taylor expansion of functions along suitable curves without assuming the existence of a basis of the underlying vector space. 
We briefly recall the definition of the Riemannian Hessian for convenience.
\begin{definition}
    Let $(M,g)$ be a Hesse-manifold and $f\colon M \to \mathbb{R}$ be a $C^2$- gaf. Then the Riemannian Hessian of $f$ at $p \in M$ is defined as follows:
    \begin{equation*}
        \Hess f(p) \colon T_pM \to T_pM \quad u\mapsto \nabla_u\nabla f.
    \end{equation*}
\end{definition}

To relate the Riemannian Hessian to local minimality, we analyze the second-order expansion of $f$ along smooth curves.
Let $c\colon I \to M$ be a smooth curve with $c(0) = p$, and define $g = f\circ c$. Since $g\colon I \to \R$ is a classical $C^2$- function, we have the standard Taylor expansion
\begin{equation}
\label{equation: second-order Taylor expansion g}
    f(c(t)) = g(t) = g(0) +tg'(0)+ \frac{t^2}{2}g''(0)+\mathcal{O}(t^3).
\end{equation}
The first derivative follows from the chain rule:
\begin{equation}
\label{equation: first derivative g}
    g'(t) = D_{c(t)}f(c'(t)) = g_{c(t)}\big( \nabla f(c(t)),c'(t)\big).
\end{equation}
In particular,
\begin{equation*}
    (f\circ c)'(0) = g_p\big(\nabla f(p),c'(0)\big). 
\end{equation*}
Thus, first-order behavior is completely determined by the Riemannian gradient.

To compute the second derivative $g''(t)$, we must differentiate $g_{c(t)}\big(\nabla f(c(t)),c'(t)\big)$.
This requires a notion of differentiation of vector fields along curves. 
Those vector fields are defined analogously to \cite[Definition 5.28.]{Bou23} as follows:
\begin{definition}
    Let $M$ be a manifold and $c\colon I \to M$ be a curve on $M$. A (smooth) map $Z\colon I \to TM$ is called a \textit{(smooth) vector field on $c$} if $Z(t) \in T_{c(t)}M$ for all $t \in I$. The set of all smooth vector fields on $c$ is denoted by $\mathcal{V}(c)$. 
\end{definition}
To make sense of differentiation of vector fields on curves, we require an appropriate operator with certain properties. Since not all vector fields $Z \in \mathcal{V}(c)$ are of the form $X \circ c$ for some $X \in \mathcal{V}(M)$, we cannot simply use the Levi-Civita connection on $M$ and must introduce a different concept for differentiating such vector fields. This is precisely where the metric spray structure becomes essential.
\begin{remark}\label{rem:fo-connection}
It is a standard argument that every connection $\nabla$ on a \textbf{finite-dimensional} vector bundle is of \emph{first order} in the sense that for section $X,Y$ and $m \in M$, the  value $\nabla_XY (m)$ depends only on the value $X(m)$ and the first order jet of $Y$. Unfortunately, the finite-dimensional proof does not generalize without further assumptions. One can prove that every connection associated to a spray, cf.  Appendix \ref{App:diffgeo}, is a first order connection in this sense. It is unknown whether there exist connections on infinite-dimenisonal manifolds which are not of first order.
\end{remark}

If the Levi–Civita connection is induced by a metric spray, then one obtains a canonical differentiation operator along curves called \textit{the covariant derivative along c}.
\begin{lemma}
\label{lem: induced covariant derivative}
    Let $(M,g)$ be a Hesse-manifold. For every smooth curve $c\colon I \to M$, there exists a unique operator 
    \begin{equation*}
        \Dt\colon \mathcal{V}(c) \to \mathcal{V}(c),
    \end{equation*}
    called the \textit{covariant derivative along c}, that
    satisfies the following properties for all $Y,Z \in \mathcal{V}(c), X \in \mathcal{V}(M), g \in C^1(I,\mathbb{R}) \text{ and } a,b \in \mathbb{R}$: \rm
    \begin{enumerate}
        \item $\mathbb{R}$-linearity: $\Dt \big(aY +bZ\big) = a\Dt Y+b\Dt Z,$
        \item Leibniz rule: $\Dt \big(gZ\big) = g'Z +g\Dt Z,$
        \item Chain rule: $\big(\Dt \big(X \circ c\big)\big)(t) = \nabla _{c'(t)}U$ for all $t \in I$.
        \item Product rule: $\dt  g(Y,Z) = g(\Dt Y,Z)+g(Y,\Dt Z),$
    \end{enumerate}
    where $g(Y,Z) \in C^1(I,\mathbb{R})$ is defined by $g(Y,Z)(t) = g_{c(t)}(Y(t),Z(t))$.
\end{lemma}
\begin{proof}
    The existence and uniqueness of such an operator follows from \cite[Proposition 4.36]{Sch23}. The construction presented there is based on the metric spray and yields a covariant derivative along curves satisfying properties (i)–(iv). 
\end{proof}
\begin{remark}
    In the finite-dimensional setting, analogous constructions to Lemma \ref{lem: induced covariant derivative} are often carried out using local frames and coordinate representations, as for instance done by Boumal in \cite[Theorem 5.29.]{Bou23}. Such arguments rely on the existence of finite-dimensional bases of the tangent spaces.
    
    In contrast, the present approach is based on the spray-induced connection and does not require the use of local frames. The differentiation operator along curves is constructed intrinsically, without resorting to basis expansions. This makes the argument directly applicable in the weak infinite-dimensional Riemannian setting considered here.
\end{remark}
To relate the Riemannian Hessian to the second-order expansion along curves, we express it in terms of the induced covariant derivative.
Let $c\colon I \to M$ be a smooth curve with $c(0)=p$ and $c'(0)=v$. By the chain rule for the induced covariant derivative along c, we obtain
\begin{equation}
    \label{equation: covariant derivative Hessian}
       \Hess f(p)[v] = \nabla_v\nabla f= \Dt \nabla f(c(t))_{| t=0}.
\end{equation}
Using the representation of the Riemannian Hessian in terms of the induced covariant derivative \eqref{equation: covariant derivative Hessian} and the structural properties established in Lemma \ref{lem: induced covariant derivative}, the computation of the second derivative of $g = f\circ c$ proceeds exactly as in the finite-dimensional case in \cite[5.9]{Bou23}. As the argument uses only structural properties of the covariant derivative, it remains valid in the present weak Riemannian framework. Hence,
\begin{equation}
\label{equation: second derivative g}
    g''(t) = g_{c(t)}\big(\Hess f(c(t))[c'(t)],c'(t)\big)+g_{c(t)}\big(\nabla f(c(t)),c''(t)\big).
\end{equation}
Consequently, the second-order Taylor expansion of $f \circ c$ is given by
\begin{equation}
\label{equation: second order Taylor expansion of f c}
   f(c(t)) = f(p) + t g_p\big(\nabla f(p),v\big) + \frac{t^2}{2}g_p\big(\Hess f(p)[v],v\big) + \frac{t^2}{2}g_p\big(\nabla f(p),c''(0)\big) + \mathcal{O}(t^3).
\end{equation}
Having expressed the second-order Taylor expansion in terms of the Riemannian gradient and the Riemannian Hessian, we now adopt the notion of second-order critical points as introduced in the finite-dimensional setting by Boumal \cite[Section 6.1]{Bou23}. These points will be shown to coincide precisely with the local minimizers of a function, if in addition the Riemannian Hessian at these points is coercive. Establishing this result relies on the second-order Taylor expansion of $f\circ c$ (cf. \eqref{equation: second order Taylor expansion of f c}).
\begin{definition}
    Let $M$ be a $C^2$- manifold and $f\colon M \to \mathbb{R}$ be a $C^2$- function. A point $p\in M$ is called a \textit{second-order critical point} for $f$ if it is a critical point and
    \begin{equation*}
        (f\circ c)''(0) \geq 0
    \end{equation*}
    for all smooth curves $c$ on $M$ such that $c(0)=p$.
\end{definition} 
In direct analogy of the finite-dimensional case \cite[Proposition 6.3.]{Bou23}, one can show, that critical points are exactly the points where the Riemannian gradient vanishes and the Riemannian Hessian is positive semi-definite. The proof carries over directly to the weak Riemannian setting, as it relies solely  on the first and second derivatives of $f\circ c$, which we have established in \eqref{equation: first derivative g} and \eqref{equation: second derivative g}.  
\begin{proposition}
\label{proposition: second order points via grad und Hessian}
    Let $f\colon M\to \mathbb{R}$ be a smooth gaf on a Hesse manifold $M$. Then, $x$ is a second-order critical point if and only if $\nabla f(x) = 0$ and  $\Hess f(x) \succeq 0$. 
\end{proposition}
We now turn to the proof of the main result. While the Riemannian gradient condition provides a necessary criterion, this theorem goes further by establishing when a critical point is indeed a minimizer. 
\begin{proposition}
\label{proposition: Hesse minimizer}
    Let $(M,g)$ be a Hesse manifold and let $f\colon M\to \mathbb{R}$ be a $C^2$-gaf. For $p \in M$, suppose that the Riemannian Hessian is coercive, i.e. there exists $\mu >0$ such that
    \begin{equation}
    \label{equation: Hessian coercivity}
        g_p(\Hess f(p)[v],v) \geq \mu \vvvert v\vvvert_p^2, \quad \forall v\in T_pM. 
    \end{equation}
    Then, any strict second-order critical point of $f$ is a strict local minimizer. 
\end{proposition}
\begin{proof}
Let $\phi\colon U_\phi \to V_\phi$ be a chart around $p$ with $\phi(p) = 0$. 
Since $V_\phi$ is an open subset of a locally convex space, there exists an open convex neighborhood 
$W_\phi \subset V_\phi$ containing $0$.

For any $x \in W_\phi$, define a smooth curve on $M$ via
$
c(t):= \phi^{-1}(t x).
$
%The chain rule yields
%\[
%c'(0) = \mathrm{D}_{\phi(p)} \phi^{-1}(x) \in T_p M.
%\]
By the second-order Taylor expansion of $f$ along $c$ (cf.~\eqref{equation: second order Taylor expansion of f c}) and the fact that $p$ is a critical point, we obtain
\[
f(c(t)) = f(p) + \frac{t^2}{2} \, g_p\bigl( \Hess f(p)[c'(0)], c'(0) \bigr) + R(t),
\]
where $R(t) = \mathcal{O}(t^3)$, i.e.\ $\lim_{t \to 0} R(t)/t^3 = 0$.
By the coercivity of the Hessian at $p$, we have
\[
g_p\bigl( \Hess f(p)[c'(0)], c'(0) \bigr)
  \;\geq\; \mu \, \vvvert c'(0)\vvvert_p^2
  \;=\; \mu \,\bigl\vvvert \mathrm{D}_{\phi(p)}\phi^{-1}(x)\bigr\vvvert_p^2,
\]
and therefore
\begin{equation}\label{eq:taylor-lower-bound}
f(c(t)) \;\geq\; f(p) + \frac{t^2 \mu}{2} \,\bigl\vvvert\mathrm{D}_{\phi(p)}\phi^{-1}(x)\bigr\vvvert_p^2 + R(t).
\end{equation}
On $E_\phi$ we define a norm as follows:
\begin{equation*}
    \vvvert x\vvvert_\phi : = \sqrt{g_p\big(D_{\phi(p)}\phi^{-1}(x),D_{\phi(p)}\phi^{-1}(x)\big)}, \quad x\in E_\phi. 
\end{equation*}
By construction, with respect to this norm, the linear mapping
\begin{equation*}D_{\phi(p)}\phi^{-1} \colon \big(E_\phi,\vvvert \cdot\vvvert_{\phi(p)}\big) \to \big(T_pM,g_p\big)
\end{equation*}
is continuous, where we identified $T_{\phi(p)}V_\phi \cong E_\phi$. Bounding by the operator norm $A>0$,
\[
\bigl\vvvert\mathrm{D}_{\phi(p)}\phi^{-1}(x)\bigr\vvvert_p^2
  \;\leq\; A^2\vvvert x\vvvert_{\phi(p)}^2
\quad\text{for all } x \in W_\phi.
\]
Since $R(t)=\mathcal{O}(t^3)$, there exists $\xi > 0$ such that
\[
|R(t)| \;\leq\; \frac{t^2}{2} \,\mu A^2
\quad\text{for all } t \in (0,\min\{1,\xi\}).
\]
Using \eqref{eq:taylor-lower-bound}, we obtain
\begin{align*}
f(c(t))
  &\geq f(p) + \frac{t^2 \mu}{2} \,A^2 + R(t) \\
  &\geq f(p) + \frac{t^2 \mu}{2} A^2 \vvvert x\vvvert_{\phi(p)}^2 - \frac{t^2 \mu}{2} A^2 
  = f(p) + \frac{t^2 \mu}{2} A^2\,(\vvvert x\vvvert_{\phi(p)}^2 - 1).
\end{align*}
Now restrict to $x \in W_\phi$ with $\vvvert x\vvvert_{\phi(p)} < 1$. Then $\vvvert x\vvvert_{\phi(p)}^2 - 1 < 0$, and thus
\[
f(c(t)) > f(p)
\quad\text{for all } t \in (0,\min\{1,\xi\}) \text{ and all } x \in W_\phi \text{ with } 0 < \vvvert x\vvvert_{\phi(p)} < 1.
\]
Define
\[
Y_\phi := \Bigl\{\phi^{-1}(t x) \,\Big|\,
          t \in (0,\min\{1,\xi\}),\; x \in W_\phi,\; \vvvert x\vvvert_{\phi(p)} < 1\Bigr\}.
\]
Since $\phi$ is a homeomorphism and the set 
$\{t x \mid t \in (0,\min\{1,\xi\}),\, x \in W_\phi,\ \vvvert x\vvvert_{\phi(p)} < 1\}$ 
is open in $V_\phi$ with respect to the locally convex topology, the set $Y_\phi$ is open in $M$. 
By the preceding estimate, we have $f(q) > f(p)$ for all $q \in Y_\phi$, so $p$ is a strict local minimizer of $f$.
\end{proof}
\begin{remark}
    The coercivity of the Riemannian Hessian represents a key difference compared to the finite-dimensional case. This is well known, see e.g. \cite{Eli71} for the use of coercivity conditions on Banach manifolds in relation to Palais and Smales condition (C). Condition (C) replaces compactness arguments which are not available in our setting. In particular, coercivity does not follow from the positive definiteness of the Riemannian Hessian and must therefore be assumed separately. 
\end{remark}
Having established first- and second-order optimality conditions on weak Riemannian manifolds, we now turn to a concrete descent method. In Section ~\ref{sec: num exp}, we will apply these optimality conditions to specific examples alongside this method.

\section{The Riemannian Gradient Descent Method}
\label{sec: RGD}
In this chapter, we introduce a basic descent method, namely the Riemannian gradient descent (RGD) algorithm, and establish convergence results for this method. 
Before we can state the algorithm, we need an auxiliary structure. In finite dimensional optimization on manifolds \cite[Chapter 3.6]{Bou23} one defines 
\begin{definition} \label{defn:boumal_retract}
A smooth map $\mathcal{R} \colon TM \rightarrow M$ is called a \emph{retraction} if for every $v \in TM$ the smooth curve $c_v(t):= \mathcal{R}(tv)$ satisfies $c(0)=x$ and $\dot{c}(0)=v$. 
\end{definition}
We deviate slightly from loc.cit.~and will allow retractions defined only on an open neighborhood $\Omega$ of the zero-section in $TM$. However, even with this relaxation, we will see that retractions are not sufficient as the next example shows.

\begin{example}\label{ex:bad:retraction}
Let $\mathbb{S}^1\subseteq \R^2$ be the unit circle. We recall from \cite[Example 3.8]{Sch23} that the diffeomorphism group $\mathrm{Diff}(\mathbb{S}^1)$ is an infinite-dimensional Lie group not modeled on a Banach space. The tangent bundle of the Lie group is trivial, \cite[Lemma 3.12 (b)]{Sch23}, i.e.\ the group multiplication $m$ induces a diffeomorphism
$$\Phi^{-1} \colon T\mathrm{Diff} (\mathbb{S}^1) \rightarrow \mathcal{V}(\mathbb{S}^1) \times \mathrm{Diff} (\mathbb{S}^1),\quad \Phi^{-1}(v_g):=(g,Dm((0_{g^{-1}},v_g)))$$
where the vector field $\mathcal{V}(M)$ is identified with the tangent space at the identity. Further, the Lie group exponential of $\mathrm{Diff}(\mathbb{S}^1$ is the map $\exp \colon \mathcal{V} (\mathbb{S}^1) \rightarrow \mathrm{Diff} (\mathbb{S}^1), X \mapsto \text{Fl}^X_1,$ sending a vector field to its time $1$-flow. Now the map 
$$\mathcal{R} \colon T\mathrm{Diff}(\mathbb{S}^1) \rightarrow \mathrm{Diff} (\mathbb{S}^1), v_g \mapsto g \circ \exp (Dm((0_{g^{-1}},v_g)))$$
is smooth and satisfies $\mathcal{R}(0_g)=g\circ \exp(0_{\mathrm{id}})=g\circ \mathrm{id} =g$. Exploiting that $Dm(0_{g^{-1}},\cdot)$ is continuous linear and $D_0\exp = \mathrm{id}_{\mathcal{V}(M)}$, the chain rule yields
$$\left.\frac{d}{dt}\right|_{t=0}\mathcal{R}(tv_g)=Dm\left(0_{g},D_0\exp \left(\left.\frac{d}{dt}\right|_{t=0} tDm(0_{g^{-1}},v_g)\right)\right)=v_g.$$
Hence $\mathcal{R}$ is a retraction, but it is well known that this retraction does not restrict to a local diffeomorphism on any zero-neighborhood in $T_g \mathrm{Diff} (\mathbb{S}^1)$ to any neighborhood of $g \in \mathrm{Diff}(\mathbb{S}^1)$. Indeed one can show, see e.g. \cite[Example 3.42]{Sch23} for details, that in any neighborhood of $g$ there are infinitely many points not in the image of $\mathcal{R}$. One can indeed even find continuous curves which intersect the image of $\mathcal{R}|_{T_g\mathrm{Diff} (\mathbb{S}^1)}$ only in $g$. A similar result holds for diffeomorphism groups of arbitrary compact manifolds of dimension $\geq 2$. 
\end{example}

Summing up, Example \ref{ex:bad:retraction} shows that the retraction condition from Definition \ref{defn:boumal_retract} will lead to mappings on manifolds whose image fails to be a neighborhood of the foot point. In other words, in infinite-dimensions the retraction property fails to give mappings allowing us to step into all directions from the footpoint. This is certainly undesirable, whence the following definition is more suitable: 

\begin{definition}
Let $M$ be a smooth manifold. Then a smooth map $\Sigma \colon TM \supseteq \Omega \rightarrow M$ defined on $\Omega$ an open neighborhood of the zero-section is called \emph{local addition} if it satisfies 
\begin{enumerate}
\item $\Sigma (0_x)=x$ for all $x \in M$,
\item the map $\theta :=(\pi_M, \Sigma) \colon \Omega \rightarrow M \times M, \theta(v_x)=(x,\Sigma(v_x))$ induces a diffeomorphism onto its open image $\theta (\Omega) \subseteq M \times M$.
\end{enumerate}
We call the local addition normalized if $D(\Sigma|_{\Omega \cap T_x M})_{0_x} =\mathrm{id}_{T_xM}$ for all $x\in M$. 
\end{definition}

Before we give examples of (nontrivial) retractions and local additions in Example \ref{Ex:loc_add_strong}, we illustrate first the relation between local additions and retractions.

\begin{lemma}\label{lem:ladd_ret}
Let $M$ be a smooth manifold. 
\begin{enumerate}
\item Every local addition $\Sigma\colon \Omega \rightarrow M$ induces a normalized local addition $\Sigma_N$ which is a retraction on $\Omega$.
\item If, in addition, $M$ is a paracompact Banach manifold, then every retraction $\mathcal{R}$ induces a normalized local addition.
\item If, in addition, $(M,g)$ is a paracompact strong Riemannian manifold, then every local addition induces a (normalized) local addition on $TM$. 
\end{enumerate}
\end{lemma}

\begin{proof}
(1) By \cite[A.14]{AaHaS20} every local addition can be modified to yield a normalized local addition $\Sigma_N \colon \Omega \rightarrow M$. Shrinking $\Omega$ we may assume without loss of generality, that $\Omega_x := T_x M \cap \Omega$ is star-shaped around $0_x$. Hence, for $v\in \Omega_x$ we have $\Sigma_N(0v)=x$ and since $\Sigma_N$ is normalized, the chain rule yields $\left.\frac{d}{dt}\right|_{t=0} \Sigma_N (tv)=v$. So $\Sigma_N$ is a retraction on $\Omega_x$ for every $x\in M$. 
(2) Let $\mathcal{R} \colon \tilde{\Omega} \rightarrow M$ be a retraction. Since $\left.\frac{d}{dt}\right|_{t=0}\mathcal{R}(tv)=v$ for all $v\in TM$ we see that the derivative of $\mathcal{R}|_{\tilde{\Omega}\cap T_xM}$ at the zero-section is the identity map. Then paracompactness and the inverse function theorem show that we can shrink $\Omega$ to an open neigborhood on which $\mathcal{R}$ restricts to a normalized local addition. The details are recorded in \cite[Lemma 3.15]{KaS25}.
(3) Finally, if we are given a  local addition $\Sigma \colon \Omega \rightarrow M$ on some open neighborhood of the zero-section, it can be extended using the argument in \cite[Lemma 10.2]{Mic80} to a (normalized) local addition on all of $TM$. 
\end{proof}
Lemma \ref{lem:ladd_ret} implies that for finite-dimensional (paracompact) manifolds normalized local additions are equivalent to retractions as in \cite{Bou23}. The point in having a retraction is that starting at $x$ we can locally reach every point near to $x$ by a suitable curve. In infinite-dimensions a local addition assures this, whence the stronger concept is preferred. 
\begin{example}\label{Ex:loc_add_strong}
Let $(M,g)$ be a strong Riemannian manifold. Then as in finite-dimensions, $M$ admits a Riemannian exponential map $\exp \colon TM \supseteq \Omega \rightarrow M$, cf. \cite[Chapter 1.6]{Kli95}. The Riemannian exponential is smooth and satisfies $D(\exp|_{\Omega \cap T_xM})_{0_x}=\mathrm{id}_x$ for $x\in M$, i.e. it is a normalized local addition (this is the standard source for retractions in finite dimensions).

For any compact manifold $K$, the set of smooth functions $C^\infty (K,M)$ can then be endowed with the structure of a Fr\'echet manifold such that $TC^\infty (K,M)\cong C^\infty (K,TM)$. Here the identification takes $T_hC^\infty (K,M) \cong \{F \in C^\infty (K,TM) \colon \pi_M \circ F = h\}$ . Further, the pushforward $\exp_\ast \colon C^\infty  (K,\Omega)\rightarrow C^\infty (K,M), \exp_\ast (g)=\exp \circ g$ is smooth. Since also the pushforwards of the associated mappings $\theta = (\pi_M,\exp)$ and $\theta^{-1}$ are smooth, we deduce that $\exp_\ast$ is a local addition. The identification of the tangent bundle yields, see \cite[2.22]{Sch23},  $D(\exp_\ast) = (D\exp)_\ast$, whence $\exp_\ast$ is a normalized local addition on $C^\infty(K,M)$. 
\end{example}
For a $C^1$- weak Riemannian manifold the Riemannian gradient descent method can be formulated as shown in Algorithm 1 below.
\begin{algorithm}
\label{algorithm: RGD}
    \caption{Riemannian Gradient Descent Method on $(M,g)$}
    \begin{flushleft}

    \textbf{Input:} $x_0 \in M$, $f\in C^1(M,\R)$, normalized local addition $\Sigma$ on $M$.\\
    
    \textbf{For} $k = 0,1,2,...$\\
    \quad pick a step-size $\alpha_k >0$ and set \\
    \quad $x_{k+1} = \Sigma_{x_k}(s_k)$ for $s_k = -\alpha_k \nabla f(x_k)$
    \end{flushleft}
\end{algorithm}

Our exposition follows the structure of Boumal \cite[Section 4.3]{Bou23}, where RGD is discussed in the finite-dimensional setting. We show that, under an additional assumption, these results carry over to the weak Riemannian setting. In particular, we show that every accumulation point of the sequence of iterates generated by Algorithm~1. is a critical point of $f$ and that the norms of the corresponding gradients converge to zero.

In order to prove this result, we require a notion of continuity for the Riemannian gradient $\nabla f$. In particular, we need $\nabla f$ to be sequentially continuous. This property cannot be inferred directly from the defining property of the Riemannian gradient, due to the incompatibility of the topologies on the tangent bundle of a weak Riemannian manifold.

In the following we will show that $\nabla f$ is sequentially continuous whenever the sequence $\big(\nabla f(p_n)\big)_{n\in \mathbb{N}}$ converges in $TM$ for a convergent sequence $(p_n)_{n\in \mathbb{N}} \subset M$. 
\begin{lemma}
    Let $(M,g)$ be a weak Riemannian $C^1$-manifold, and let $\big(p_n\big)_{n\in \mathbb{N}}\subset M$ be a sequence converging to $p\in M$. Let $f\colon M \to \mathbb{R}$ be a gaf such that the sequence $\big(\nabla f(p_n)\big)_{n \in \mathbb{N}}$ converges in $TM$, then 
    \begin{equation*}
        \lim \limits_{n\to \infty} \nabla f(p_n) = \nabla f(p).
    \end{equation*}
\end{lemma}
\begin{proof}
     %As $\lim \limits_{n\to \infty} \nabla f(p_n) \in TM$, there exists $q \in M$ such that $\lim \limits_{n\to \infty} \nabla f(p_n) \in T_qM.$ To prove the statement, we begin by showing that $p = q$. Next, we demonstrate that $\lim \limits_{n \to \infty} \nabla f (p_n) = \nabla f(q)$. Together this implies $\lim \limits_{n\to \infty} \nabla f(p_n) = \nabla f(p)$.
     
     Since $\left( \nabla f(p_n)\right)_{n\in \mathbb{N}}$ converges in $TM$ and $\pi _M$ is continuous, it follows that 
     \begin{equation*}
        \lim \limits_{n \to \infty} \pi _M\big ( \nabla f(p_n)\big) = \pi_M\big(\lim \limits_{n\to \infty}\nabla f(p_n)\big) = p.
    \end{equation*}
    We localize in a chart $(\phi,U)$ of $M$ around $p$. So without loss of generality, $TU= U \times E$ (suppressing the identification). As $g$ and $Df$ are continuous, we obtain $\forall v\in T_pM$
    \begin{align*}
         g_p(\nabla f(p), v) = Df(v) = \lim \limits_{n\to \infty}D_{p_n}f(v) = \lim \limits_{n\to \infty}g_{p_n}(\nabla f(p_n), v)= g_p(\lim \limits_{n \to \infty}\nabla f(p_n), v), 
     \end{align*}
     Since $g_p$ is nondegenerate we conclude that $\lim \limits_{n\to \infty} \nabla f (p_n) = \nabla f (p)$.
\end{proof}
With this result, the sequential continuity of the Riemannian gradient can now be defined solely by requiring that the Riemannian gradients of convergent sequences converge within the tangent bundle.  
\begin{corollary}
\label{corollary: Riemannian gradient sequentially cts}
    Let $\big(M,g\big)$ be a weak Riemannian $C^r$-manifold, $r\geq 1$ and let $f\colon M \to \mathbb{R}$ be a gaf. If for all $(p_n)_{n\in \mathbb{N}}\subset M$ that converge in $M$, $f$ is such that $\lim \limits_{n\to \infty}\nabla f(p_n) \in TM$, then $\nabla f$ is sequentially continuous.
\end{corollary}
Equipped with this result, we can establish the main result of this section under the following assumptions.

\begin{assumption}
\label{assumption: global minimum}
    There exists $f_{low} \in \mathbb{R}$ such that $f(p) \geq f_{low}$ for all $p \in M$.
\end{assumption}
\begin{assumption}
\label{assumption: iteration decrease}
    At each iteration, the algorithm achieves sufficient decrease for $f$, in that there exists a constant $\rho>0$ such that, for all $k$,
    \begin{equation}
        f(p_k)-f(p_{k+1})\geq \rho \vvvert \nabla f(p_k)\vvvert^2_{p_k}
    \end{equation}
\end{assumption}
\begin{assumption}
\label{assumption: gradient convergence}
    For every sequence $(p_n)_{n\in \mathbb{N}}\subset M$ that is convergent in $M$, $\big(\nabla f(p_n)\big)_{n\in \mathbb{N}}$ converges in $TM$.
\end{assumption}
\begin{proposition}
\label{proposition: RGD gradient convergence}
    Let $f$ be a $C^1$-function satisfying \Cref{assumption: gradient convergence} and \Cref{assumption: global minimum} on a weak Riemannian $C^r$-manifold, $r\geq1$. Let $p_0,p_1,p_2,...$ be iterates satisfying \Cref{assumption: iteration decrease} with constant $\rho$. Then
    \begin{equation*}
        \lim \limits_{n\to \infty}\vvvert\nabla f(p_n)\vvvert_{p_n} = 0.
    \end{equation*}
    In particular, all accumulation points are critical points. Furthermore, for all $K\geq 1$, there exists $k \in \{0,...,K-1\}$ such that 
    \begin{equation*}
        \vvvert\nabla f(p_k)\vvvert_{p_k}\leq \sqrt{\frac{f(p_0)-f_{low}}{\rho}}\frac{1}{\sqrt{K}}.
    \end{equation*}
\end{proposition}
\begin{proof}
    The proof proceeds analogously to that in \cite[4.7.]{Bou23}, relying on a telescoping sum argument together with the sequential continuity of $\nabla f$ and $\vvvert \cdot \vvvert$. Consequently, it extends directly to the weak Riemannian setting.  
\end{proof}
\begin{remark}
Assumption \Cref{assumption: global minimum} and \Cref{assumption: iteration decrease} are standard assumptions known from finite-dimensional Riemannian optimization. The proof in \cite[4.7.]{Bou23} shows that Assumption \Cref{assumption: global minimum} and \Cref{assumption: iteration decrease} are sufficient to guarantee that the norm of the Riemannian gradient along the iteration sequence converges to zero. However, in the infinite-dimensional setting we additionally require the sequential continuity of the Riemannian gradient, ensured by Assumption \Cref{assumption: gradient convergence}, in order to conclude that all accumulation points are critical points. In the next example, however, we will see that Assumption \Cref{assumption: gradient convergence} is not guaranteed apriori in the infinite-dimensional setting. 
\end{remark}
\begin{example}\label{ex:5.10}
    We consider the length function on the space $C^\infty(\SSS^1,\R^2)$ 
    \[\mathcal{L} \colon C^\infty(\SSS^1,\R^2) \rightarrow \R, \mathcal{L}(c):=\int_{\SSS^1} \lVert \dot{c}\rVert\, d\mu.\]
    The space $C^\infty(\SSS^1,\R^2)$, viewed as a locally convex space equipped with the weak Riemannian metric $g^{L^2}(h,k) = \int_{\SSS^1}\langle h,k\rangle d \mu $ (Example \ref{ex:basic}), forms a weak Riemannian manifold. An explicit formula for the gradient of $\mathcal{L}$ can be computed for curves $c\in \text{Imm}(\SSS^1,\R^2)\subset C^\infty (\SSS^1,\R^2)$ as follows:
    \begin{equation*}
    \nabla \mathcal{L}(c) = -k_cN_c \lVert\dot{c}\rVert \in C^\infty (\SSS^1,\R^2),
\end{equation*}
where $N_c (z)= (-y_z (z),x_z(z))^\top$ denotes the normal vector to the curve $c(z)=(x(z),y(z))$ and $k_c$ its signed curvature. The computation is essentially the same as in Example \ref{example: Riemannian Gradient on Imm} below, where the factor $\lVert\dot{c}\rVert$ has to be included i the gradient.
We emphasize that this expression is only well-defined for immersions, since the signed curvature $k_c$ requires a nonvanishing derivative of $c$ and is undefined for points where $\dot{c} = 0$. In particular, for curves that leave the space of Immersions, the curvature-based Riemannian gradient no longer exists in a classical sense.

We define a sequence $(c_k)_{k\in \mathbb{N}} \subset \text{Imm}(\SSS^1,\R^2)$ by
\begin{equation*}
    c_k = \frac{(-1)^k}{k}\text{id}_{\mathbb{S}^1}, \quad k\in \mathbb{N}.
\end{equation*}
Observe that for $c =r\cdot\text{id}_{\SSS^1}$ for some $r\neq 0$, the Riemannian gradient of $\mathcal{L}$ at $c \in \text{Imm}(\SSS^1,\R^2)$ is given by
\begin{equation*}
    \nabla \mathcal{L}(r\cdot \text{id}_{\SSS^1}) = - \text{sgn}(r)\text{id}_{\SSS^1}
\end{equation*}
Clearly, $c_k \rightarrow 0$ as $k \rightarrow \infty$, and thus $(c_k)_{k\in \mathbb{N}}$ converges within $C^\infty(\SSS^1,\R^2)$. Nevertheless, since $\nabla \mathcal{L}(c_n) = (-1)^{n+1}\text{id}_{\SSS^1}$, the sequence of Riemannian gradients $\big(\nabla \mathcal{L}(c_n)\big)_{n\in \mathbb{N}}$ does not converge within $TM$.
\end{example}
 
 %This example shows that, in the infinite-dimensional setting, Assumption \ref{assumption: gradient convergence} does not hold automatically and therefore has to be imposed.
 
 \begin{remark}
     Observe that Assumption \ref{assumption: iteration decrease}, which imposes a sufficient decrease condition, depends indirectly on the choice of a (normalized) local addition $\Sigma$. In this paper, we do not further address the selection of step sizes or the construction of local additions that satisfy this assumption; this is deferred to future work. 
     
     Local additions on weak Riemannian manifolds present additional challenges: For example, convergence of typically depends on geodesic completeness or the iterates forming a Cauchy-sequences in the length metric, cf. e.g. \cite[Section 4.6]{Bou23}. In general, the length metric on a weak Riemannian manifold fails to be a metric. However, due to \cite{BaMaW25} the situation is much better for certain weak Riemannian metrics modeled on Hilbert spaces. 
     
     Provided that a suitable local addition exists, one may expect an analogue of a result from the finite-dimensional setting if one makes a {\em Lipschitz-type} assumption on the regularity of the gradient, see \cite[4.4]{Bou23}.
 \end{remark}
% \begin{assumption}
% \label{assumption: subset retraction}
%     For a given subset $S \subseteq TM$, there exists a constant $L >0$ such that for all $s_p \in S$,
%     \begin{equation*}
%         f\big(R_p(s_p)\big) \leq f(p) + g_p\big(\nabla f(p), s_p\big) + \frac{L}{2}\|s_p\|^2.
%     \end{equation*}
% \end{assumption}
% \begin{proposition}
%     Let $\big(M,g\big)$ be a weak Riemannian $C^1$- manifold and let $f\colon M \to \mathbb{R}$ be a gaf. For a retraction $R$, let $f\circ R$ satisfy \ref{assumption: subset retraction} on a set $S\subseteq TM$ with constant $L$. If the pairs $(p_k,s_k)_{k\in \mathbb{N}}$ generated by the RGD with step-sizes
%     \begin{equation*}
%         \alpha_k \in[\alpha_{\min},\alpha_{\max}] \subset \big(0,\frac{2}{L}\big)
%     \end{equation*}
%     all lie in $S$, then the algorithm produces sufficient decrease, hence \ref{assumption: iteration decrease} holds with
%     \begin{equation*}
%         c = \min\bigg(\alpha_{\min} -\frac{L}{2}\alpha_{\min}^2, \alpha_{\max}-\frac{L}{2}\alpha_{\max}^2\bigg) >0.
%     \end{equation*}
%     For constant step-sizes $\alpha = \frac{1}{L}$ we have $c = \frac{1}{2L}$.
% \end{proposition}

\section{Classes of Hesse manifolds and their Optimization-relevant properties}
\label{sec: Hesse mfds}
In the preceding sections, we established first-order and second-order optimality conditions for weak Riemannian manifolds and analyzed the Riemannian gradient descent method together with its convergence properties. Although our framework is formulated for general weak Riemannian manifolds, we imposed additional structural assumptions to ensure that these optimization results hold. This led to the notion of a Hesse manifold, which is a weak Riemannian manifold endowed with extra properties that make Riemannian optimization well defined and analytically tractable. Recall from Definition \ref{definition: Hesse manifold} that a Hesse manifold is a weak Riemannian manifold which admits a metric spray.

In this chapter, we present two important classes of Hesse manifolds and investigate both their fundamental geometric features and their optimization-related properties. Our primary focus will be on the robust Riemannian manifolds. We then turn to the more classical strong Riemannian manifolds. The results derived in this chapter are theoretical and possibly not of immediate interest to optimizers. However, they are new to the best of our knowledge and provide a convenient framework. In particular, the typical Riemannian metrics from shape analysis give rise to Hesse manifolds.

\subsection{Robust Riemannian manifolds}
An important class of weak Riemannian manifolds that are suitable for optimization purposes, yet do not qualify as strong Riemannian manifolds, consists of robust Riemannian manifolds, as they possess a Levi-Civita connection by definition. We next examine their geometric structure, provide concrete examples, and characterize when a weak Riemannian manifold qualifies as robust.

Robust Riemannian manifolds were introduced by Micheli and collaborators in \cite{MMM13}. This strengthening of the notion of a weak Riemannian metric allows for example curvature calculations for Riemannian submersions.

\begin{definition}\label{defn_robust_riem_mfd}
Let $(M,g)$ be a weak Riemannian manifold. We say $g$ is a \emph{robust Riemannian metric} if 
\begin{enumerate}
\item The Hilbert space completions of the fibres $\overline{T_xM}^{g_x}$ with respect to the inner product $g_x$ form a smooth vector bundle $\overline{TM} = \bigcup_{x \in M} \overline{T_xM}^{g_x}$ over $M$ whose trivializations extend the bundle trivializations of $TM$.
\item the metric derivative of $g$ exists. 
\end{enumerate}
A weak Riemannian manifold with a robust Riemannian metric will be called a \emph{robust Riemannian manifold}.
\end{definition}
\begin{remark}
Note that condition (1) in Definition \ref{defn_robust_riem_mfd} entails that the inner products $g_x$ induced by the weak Riemannian metric are locally (in a chart) equivalent to each other and thus induce the same Hilbert space completion of the fibres $T_xM$. 
\end{remark}
Before we consider examples of robust Riemannian metrics, let us first assert that:
\begin{proposition}\label{prop:rob_is_hesse}
Every robust Riemannian manifold $(M,g)$ is a Hesse manifold.
\end{proposition}

\begin{proof}
By property (1) of a robust Riemannian manifold, $\overline{TM}\rightarrow M$ is a Hilbert bundle over $M$ with typical fibre $H$. Further the Riemannian metric $g$ induces a Riemannian bundle metric $\overline{g}$ on $\overline{TM}$ (the distinction here is that $\overline{TM}$ is not the tangent bundle of $M$). We work locally on a chart domain $U$ (but suppress the chart in the notation and also the identification $\overline{TU} \subseteq \overline{TM}$). For every point $x \in U$, $\overline{g}_U(x,\cdot)$ induces the musical isomorphisms between the Hilbert space $H$ and its dual. Hence, the formula \eqref{eq:mspray} yields a well defined quadratic form $\Gamma_U (x,\cdot) \colon H \rightarrow H$ which smoothly depends on $x\in U$. Using the polarization identity $B_U (x,v,w):=\frac{1}{2}\left(\Gamma_U(x,v+w)-\Gamma_U(x,v)-\Gamma_U(x,w)\right)$ we obtain a bilinear. Now as in \eqref{cod:loc_spray} we obtain a (linear) connection (see \cite[VII.3]{GaHaV73} or \cite[1.5]{Kli95}, neither of \cite{Lang01,Sch23} define connections on vector bundles) on $\overline{TM}$ \begin{align}\label{form:connection}\overline{\nabla}_U \colon \Gamma(TU) \times \Gamma(\overline{TU}) \rightarrow \Gamma(\overline{TU}), \quad \overline{\nabla}_U (\xi,\sigma)(x):=d\sigma (x;\xi(x))-B_U(x,\xi(x),\sigma(x)),\end{align}
i.e. $\overline{\nabla}_U$ is tensorial in $\xi$ and a derivation in $\sigma$.  
As in the proof of \cite[VIII \S 4, Theorem 4.2]{Lang01} a direct calculation shows that $\overline{\nabla}_U$ is a metric connection (cf. \cite[Definition 4.2.1]{Jost17}) in the sense that it satisfies the product rule 
\begin{align}\label{prod_rule_met_con}\xi. \overline{g}_U (\sigma, \tau) = \overline{g}_U (\overline{\nabla}_U (\xi,\sigma),\tau)+\overline{g}_U (\sigma,\overline{\nabla}_U (\xi,\tau)),\qquad \xi \in \Gamma(TU), \sigma , \tau \in \Gamma(\overline{TU})\end{align}
By property (2) of a robust Riemannian manifold, the metric derivative $\nabla$ for $g$ exists on $TM$, i.e $\nabla$. The covariant derivative $\nabla$ will be a metric derivative if on every chart domain $U$ the product rule \eqref{eq:metric_deriv} holds (for $g_U$ and $\nabla_U$). As $TU \rightarrow \overline{TU}$ pulls back the Riemannian bundle metric $\overline{g}_U$ to $g_U$, the pullback of the metric connection $\overline{\nabla}_U$ becomes the (representative of the) metric derivative $\nabla_U$ (see \cite[Proposition 5.6 (a) and Exercise 5.4]{Lee97}). In particular, $\nabla_U$ is given by the formula \eqref{form:connection}. However, rearranging \eqref{form:connection} with $\Gamma_U(x,v)=B_U(x,v,v)$ for $\xi,\sigma \in \Gamma(TU)$ implies that
\[
\overline{S}_U \colon TU \rightarrow T\overline{TU},\quad \overline{S}(x,\xi):=(x,\xi,\xi,\Gamma_U (x,\xi)) 
\] 
factors through a spray $S_U \colon TU \rightarrow T(TU)$ ($\subseteq T\overline{TU}$ via the tangent of $TU\rightarrow \overline{TU}$). We conclude that $\nabla_U$ is induced by $S_U$.  Thus (cf.\ \cite[VIII \S 4 Theorem 4.2]{Lang01}) $S_U$ is a metric spray for $g_U$. 
The $S_U$ are compatible under change of trivialization as in  \cite[VIII \S 4 Theorem 4.2]{Lang01}, whence they induce a metric spray of $g$.
\end{proof}

\begin{remark}
The proof of Proposition \ref{prop:rob_is_hesse} shows that one can construct Christoffel-symbol-like objects on the completion which restrict to the metric spray. A subtle point is nevertheless the interplay between spray and metric derivative. As $M$ is not even a Banach manifold, the connection \eqref{form:connection} needs to avoid a definition via (sections of) the cotangent bundle. Fortunately, the calculations in \cite{Lang01} we needed to appeal to do not need duality or cotangent bundle arguments.
\end{remark}
In \cite[p.9]{MMM13}, the authors point out (but do not give details) that the space $\mathrm{Emb}(M,N)$ of smooth embeddings with the Sobolev $H^s$-metric (for $s$ above the critical Sobolev exponent) is a robust Riemannian manifold. Further, the following was proved in \cite[Theorem 5.1]{Mic15} and yields another main class examples:
\begin{example}
Let $G$ be a possibly infinite-dimensional Lie group. Recall from \cite[Chapter 3]{Sch23} that an infinite-dimensional Lie group is called regular (in the sense of Milnor) if the so called Lie-type differential equations can be solved on $G$ (every Banach Lie group is regular). If $g$ is a right-invariant weak Riemannian metric on the regular Lie
group $G$ which admits a metric derivative, then $(G,g)$ is already a robust Riemannian manifold.
\end{example}
The following Lemma yields another class of examples which is elementary and at the same time of interest in applications. To our knowledge, the following result has not appeared with a detailed exposition in the literature before:

\begin{proposition}\label{prop:L2_robust}
Let $(H,\langle \cdot , \cdot \rangle)$ be a Hilbert space and $\Omega \subseteq H$ open. For every compact manifold $K$, the $L^2$-metric (cf. Example \ref{ex:basic}) is a robust Riemannian metric on $C^\infty (K, \Omega)$.
\end{proposition}

\begin{proof}
Note that we endow $\Omega$ with the Riemannian metric induced by the inclusion $\Omega \subseteq H$ and that the function space $K_\Omega := C^\infty (K,\Omega)$ is an open subset of the Frechet space $C^\infty (K,H)$, whence an infinite-dimensional manifold. Moreover using standard identifications \cite[Chapter 2]{Sch23}, the tangent bundle is trivial with
$TK_\Omega \cong C^\infty (K, T \Omega) \cong K_\Omega \times C^\infty (K,H)$.

Now due to \cite[Proposition 5.8]{Sch23} the metric derivative of the $L^2$-metric exists. 
The Hilbert space completion of $C^\infty (K,H)$ is the space $L^2 (K,H)$ of all (equivalence classes of) $L^2$-functions from $K$ to $H$ (cf. e.g. \cite{Pal68}). Since the bundle $TK_\Omega$ is trivial, the (fibre-wise) completion $\overline{TK_\Omega}^{L_2} \cong K_\Omega \times L^2 (K,H)$ is a bundle over $K_\Omega$ which extends $TK_\Omega$.
\end{proof}

\begin{remark}\label{rem:elastic}
An important special case of Proposition \ref{prop:L2_robust} is the case where $K=\SSS^1$ and $\Omega =\R^2\setminus \{0\} \subseteq \R^2$. Then the robust Riemannian manifold $C^\infty (\SSS^1, \R^2 \setminus \{0\})$ with the $L^2$-metric is isometrically isomorphic to the manifold \[\mathrm{Imm}_0(\SSS^1,\R^2):= \{f \colon \SSS^1 \rightarrow \R^2 \text{ is an immersion}\colon f(e^{\mathrm{i}0})=0\}\]
with a so-called elastic metric. The isometry is the so-called square-root-velocity-transform (SRVT), cf. \cite{BaCaKaKaNaP24}, and we remark that the elastic metric is invariant under the canonical action of $\mathrm{Diff}(\SSS^1)$. For this reason, the elastic metric is used in shape analysis, see e.g. \cite[Chapter 5]{Sch23} for an overview.
We note that Proposition \ref{prop:L2_robust} immediately implies that the elastic metric is a robust Riemannian metric. 

As discussed in \cite{BaCaKaKaNaP24}, the square-root-velocity-transform is just a special case of a more general family of transformations turning elastic metrics for other choices of the elastic parameters into (variants of) the $L^2$-metric. A similar analysis as in Proposition \ref{prop:L2_robust} should show that these metrics are also robust, but we will not explore this in the current paper. 
\end{remark}

Recall that due to the Nash-embedding theorem, every finite dimensional smooth Riemannian manifold $(M,g)$ admits an isometric embedding $\theta \colon (M,g) \rightarrow (\R^N,\langle \cdot , \cdot \rangle)$ for some $N$. As the pushforward $\theta_\ast \colon C^\infty (K,M)\rightarrow C^\infty (K,\R^N), \theta_\ast (f)=\theta\circ f$ is smooth by \cite[Corollary 2.19]{Sch23}, together with the identification $TC^\infty (K,M)\cong C^\infty (K,TM)$ the map $\theta_\ast$ induces a Riemannian embedding into $C^\infty (K,\R^N)$. Thus  the following is now an immediate consequence of Proposition \ref{prop:L2_robust}:

\begin{corollary}
    For every finite dimensional Riemannian manifold $M$ and every compact manifold $K$, the $L^2$-metric turns $C^\infty (K,M)$ into a robust Riemannian manifold.
\end{corollary}

In general we lack a global isometric embedding for infinite-dimensional strong Riemannian manifolds (albeit many infinite dimensional manifolds embedd as open subsets of Hilbert spaces, cf. \cite{Hen70}). One could argue using localization arguments in charts to obtain a similar result for mapping spaces into strong Riemannian manifolds. We shall not give a detailed account of this. A first step towards this is the following Lemma, which is of interest in its own right.

\begin{lemma}\label{lem:L2_openHilb_strong}
Let $\Omega \subseteq H$ be an open subset of the Hilbert space $(H,\langle \cdot ,\cdot \rangle)$ endowed with a strong Riemannian metric $g$. For a compact manifold $K$, Write $K_\Omega := C^\infty(K,\Omega)$ for the manifold endowed with $G$, the $L^2$-metric with respect to $g$. 
\begin{enumerate}
\item 
There is a bundle trivialization 
$\Theta \colon TK_\Omega \rightarrow K_\Omega \times C^\infty (K,H)$ which takes the $G$-inner product fibre-wise to the $L^2$-metric with respect to $\langle \cdot , \cdot\rangle$.
\item $C^\infty (K,\Omega), L^2_g)$ is a robust Riemannian manifold.
\end{enumerate}
\end{lemma}

\begin{proof}
Identify $TC^\infty (K,\Omega)\cong C^\infty (K,T\Omega)\cong K_\Omega \times C^\infty (K,H)$.

(1) Recall from \cite[VII, Theorem 3.1]{Lang01} that since $g$ is a strong Riemannian metric there is a smooth map
$B \colon \Omega \times H \rightarrow H, B_p := B(p,\cdot)$ such that for every $p \in \Omega, B_p$ is a positive definite invertible operator with $g_p(u,v)=\langle B_pu,B_pv\rangle, u,v \in H$. We define
\[\theta \colon K_\Omega \times C^\infty (K,H)\rightarrow C^\infty (K,H), (f,\varphi) \mapsto B\circ (f,\varphi)\]
By construction $\theta_f :=\theta (f,\cdot)$ is bijective, linear and fibre-wise an isometry as\textit{} 
\begin{align*}\int_{\SSS^1} \langle \theta_f (\varphi), \theta_f (\psi)\rangle \, \mathrm{d}\mu &=\int_{\SSS^1} \langle B_{f(p)}  (\varphi(p)), B_{f(p)}(\psi(p))\rangle \, \mathrm{d}\mu(x)=\int_{\SSS^1} g_{f(p)} (\varphi(p), \psi(p)) \, \mathrm{d}\mu(p)\\ &=G_f (\varphi,\psi).\end{align*}
If $\theta$ is smooth, then $\Theta =(\mathrm{id}_{K_\Omega},\theta)$ satisfies the conditions in (1). To see that $\theta$ is smooth, recall that by the exponential law \cite[Theorem 2.12]{Sch23}, $\theta$ is smooth if and only if the adjoint map $\theta^\wedge \colon K_\Omega \times C^\infty (K,H) \times K \rightarrow H$ is smooth, but this map can be written as
$\theta^\wedge (f,\varphi,k)= \ev(B(\ev(f,k),\ev(\varphi(k)))$ and since $B$ is smooth and the evaluation maps of the spaces $K_\Omega$ and $C^\infty (K,H)$ are smooth, \cite[Lemma 2.16 (a)]{Sch23}, we deduce that $\theta$ is smooth.

(2) By part (1), $\Theta$ is a bundle isomorphism over the identity onto a trivial bundle. By Proposition \ref{prop:L2_robust}, $K_\Omega$ with the $L^2$-metric is a robust Riemannian manifold.
We note that as $\Theta$ induces fiber-wise an isometry, it extends in every fiber to an isometry of the Hilbert space completions (see \cite[Lemma 4.16]{Rud87}). Hence taking fibre-wise the continuous linear extensions to the completions of the fibre-maps of $\Theta$ we obtain a fibre-wise isometry
\[\overline{\Theta} \colon \sqcup_{f \in K_\Omega} \overline{T_fK_\Omega}^{g_f} \rightarrow K_\Omega \times L^2 (K,H).\]
Thus there is a unique vector bundle structure on the union of the completed spaces, making $\overline{\Theta}$ a bundle isomorphism and by construction this bundle extends $TK_\Omega$. 
The metric derivative exists again in this setting by \cite[Theorem 5.8]{Sch23}.
We conclude that $L^2_g$-is a robust Riemannian metric. 
\end{proof}

In general, the construction in part (2) of Lemma \ref{lem:L2_openHilb_strong} already hints at permanence properties of various objects connected to Riemannian metrics which are hardly surprising. However, we state them here and supply the necessary details for the proofs for the readers convenience. In particular, while it is somewhat obvious that these constructions should work, the added details should convince the reader that the constructions do not depend on the manifolds being finite-dimensional or strong manifolds.

\begin{proposition}
Let $(M,g), (N,\tilde{g})$ be weak Riemannian manifolds together with a Riemannian isometry $F\colon M \rightarrow N$ (i.e. a diffeomorphism such that $F^\ast \tilde{g}=g$). Then
$(M,g)$ is a robust Riemannian manifold if and only if $(N,\tilde{g})$ is a robust Riemannian manifold.  
\end{proposition}

\begin{proof}
Since $F$ is a Riemannian isometry, the same holds for $F^{-1}$. Clearly the situation is symmetric, so it suffices to assume that $(N,\tilde{g})$ is a robust Riemannian manifold and we shall prove that $(M,g)$ is robust.

For the completion of the bundle $TM$ we just note that the isometries $TF \colon TM\rightarrow TN$ and $TF^{-1}\colon TN\rightarrow TM$ extend fiber-wise to isometries of the Hilbert completions with respect to the inner products induced by the Riemannian metrics (see \cite[Lemma 4.16]{Rud87}).

As $F$ is a diffeomorphism, every vector field $X$ on $M$ is $f$-related to the  push-forward $\tilde{X} =F_* X := TF \circ X \circ F^{-1}$ on $N$. Now $(N,\tilde{g})$ admits a metric derivative $\tilde{\nabla}$ and we use it to define a mapping $\nabla \colon \mathcal{V}(M)^2 \rightarrow \mathcal{V}(M)$ via the formula
\[\nabla_Y Z=(F^{-1})_\ast(\tilde{\nabla}_{\tilde{Y}}(\tilde{Z}))=TF^{-1} (\tilde{\nabla}_{TF\circ Y\circ F^{-1}} (TF\circ Z\circ F^{-1})\circ F.\]
The usual finite dimensional proof, see \cite[Proposition 5.6 (a) and Exercise 5.4]{Lee97} shows that $\nabla$ is a connection  compatible with the metric, i.e. a metric derivative.
\end{proof}
\subsection{Strong Riemannian manifolds}
 We now turn to strong Riemannian manifolds, which are well established both in geometric theory and optimization. However, it should be pointed out that there are significant differences on the level of Riemannian geometry.
 \begin{example}\label{ex:GM_ellipsoid}
     Every Hilbert space is a strong Riemannian manifold as are embedded submanifolds like the unit sphere. Moreover, in the Hilbert space $\ell^2$ of square summable sequences, if we define $a_1=1$ and $a_n=1+2^{-n}, n \geq 2$, then the set
     $$E:=\left\{(x_n)_{n\in \mathbb{N}} \in \ell^2 \colon \sum_{n\in \mathbb{N}} \frac{x_n^2}{a_n^2} =1\right\},$$
     is a strong Riemannian manifold with the pullback metric. This is known as Grossmann's ellipsoid, and one can prove that while it is geodesically complete, there are points which do not admit a minimal geodesic path between them (in other words, the Hopf-Rinow-theorem fails on strong Riemannian manifolds), see \cite[4.43]{Sch23} for details.
 \end{example}
The underlying Hilbert space structure of a strong Riemannian manifold and the corresponding structure on the tangent bundles, enables direct transfer to many results from finite-dimensional optimization.
We briefly illustrate this in our setting.
 By \cite[4.5]{Sch23}, a strong Riemannian manifold can equivalently be described as follows:
\begin{lemma}
\label{definition:alternative strong Riemannian manifold}
    Let $(M,g)$ be a weak Riemannian manifold. If $M$ is a Hilbert manifold, i.e. modeled on Hilbert spaces and the injective linear map
    \begin{equation*}
        \flat \colon TM \to T^*M,\quad T_pM \ni v \mapsto g_p(v,\cdot)
    \end{equation*}
    is a vector bundle isomorphism, then $(M,g)$ is a strong Riemannian manifold.
\end{lemma}
The usual sources \cite{Lang01,Kli95} for Riemannian geometry in infinite-dimensional spaces deal with strong Riemannian manifolds. In particular, they show that the Levi-Civita derivative and the metric spray (cf.\ Appendix \ref{App:diffgeo}) exist for these manifolds. Summing up this shows the following.

\begin{lemma}
Every strong Riemannian manifold is a robust Riemannian manifold and thus a Hesse manifold.
In particular, every finite-dimensional manifold is a strong Riemannian manifold.\end{lemma}

The geometric structure of a strong Riemannian manifold guarantees the existence and the continuity of the Riemannian gradient through its unique representation.
\begin{lemma}
\label{lemma: Strong Riem. Grad. Seq. cts.}
    Let $(M,g)$ be a strong Riemannian $C^1$-manifold and $f\colon M \to \mathbb{R}$ be a $C^1$-function. Then the Riemannian gradient $\nabla f$ exists and is sequentially continuous. 
    \begin{proof}
        As $(M,g)$ is a strong Riemannian manifold, $\flat \colon TM \to T^*M$ is an isomorphism. Hence, the Riemannian gradient of any $C^1$-function $f \colon M \to \mathbb{R}$ is given by
        \begin{equation*}
            \nabla f(p) = \flat^{-1}(df(p;\cdot)).
        \end{equation*}
        By \cite[4.4]{Sch23}, $\flat$ is a bounded linear operator and thus continuous. This implies that for every sequence $(p_n)_{n \in \mathbb{N}} \subset M$ with $\lim \limits_{n \to \infty}p_n = p \in M$, that
        \begin{equation*}
            \lim \limits_{n \to \infty}\nabla f(p_n) = \lim\limits_{n\to \infty} \flat^{-1}(df(p_n;\cdot) = \flat^{-1}(\lim \limits_{n\to \infty}df(p_n;\cdot)) = \flat^{-1}(df(p;\cdot)) = \nabla f(p).\qedhere
        \end{equation*}
    \end{proof}
\end{lemma}
Consequently, on strong Riemannian manifolds, every $C^1$-function is a gaf, and Assumption \ref{assumption: gradient convergence} holds automatically. Thus, Proposition \ref{proposition: RGD gradient convergence} simplifies to:
\begin{corollary}
\label{corollary: RGD convergence on strong mfds}
    Let $(M,g)$ be a strong Riemannian $C^1$-manifold and $f$ a $C^1$-function on $M$ satisfying \ref{assumption: global minimum}. Let $p_0,p_1,p_2,...$ be iterates satisfying \ref{assumption: iteration decrease} with constant $\rho$. Then 
     \begin{equation*}
        \lim \limits_{n\to \infty}\|\nabla f(p_n)\| = 0.
    \end{equation*}
    In particular, all accumulation points are critical points. Furthermore, for all $K\geq 1$, there exists $k \in \{0,...,K-1\}$ such that 
    \begin{equation*}
        \|\nabla f(p_k)\|_{p_k}\leq \sqrt{\frac{f(p_0)-f_{low}}{\rho}}\frac{1}{\sqrt{K}}.
    \end{equation*}
\end{corollary}

Thus, combined with Lemma \ref{lemma: Strong Riem. Grad. Seq. cts.}, this implies that on strong Riemannian $C^\infty$-manifolds, the Riemannian Hessian exists for every $C^2$-function and is moreover continuous.

Although many concepts from finite-dimensional Riemannian optimization extend in an essentially analogous way to strong Riemannian manifolds, this analogy breaks down at the level of second-order optimality conditions, since even on strong Riemannian manifolds positive definiteness does not imply a coercivity condition.

\section{Computation of the Riemannian gradient and the Riemannian Hessian}
\label{sec: computation}
In this chapter, we examine the computation of the Riemannian gradient and the Riemannian Hessian. We first establish the extension property of the Riemannian gradient and the Riemannian Hessian. We then compute these objects explicitly for concrete examples. Note first that the constructions are stable under restrictions to open subsets.
\begin{lemma}
\label{lemma: Extension Riemannian Gradient}
    Let $\big(E,\langle\cdot,\cdot\rangle \big)$ be a locally convex space with a continuous inner product. Consider any open subset $M \subseteq E$. Equipped with the induced metric $g$, $\big(M,g\big)$ is a weak Riemannian manifold. Let $f\colon M \to \mathbb{R}$ be a $C^1$-function and assume that $f$ extends to a gaf $\overline{f}\colon E \rightarrow \mathbb{R}$. Then $f$ is a gaf and $\operatorname{grad}f|_M = \nabla f$, and $\nabla f$ is sequentially continuous.
\end{lemma}
The proof follows immediately from untangling the identifications and it extends to the Riemannian Hessian, i.e.:
\begin{lemma}
\label{lem: Extension Riemannian Hessian}
    In the setting of Lemma \ref{lemma: Extension Riemannian Gradient}, assume that $\big(E,\langle \cdot, \cdot\rangle \big)$ admits a spray-induced Levi-Civita connection $\nabla$. Then, the Riemannian Hessian of $f$ on $\big(M,\langle \cdot,\cdot \rangle \big)$ coincides with its ambient extension:
    \begin{equation*}
        \Hess f(p) = \Hess \overline{f}(p), \quad p\in M.
    \end{equation*}
\end{lemma}
\begin{proof}
    Since the Levi-Civita connection on $\big(M,\langle \cdot,\cdot\rangle\big)$ is the restriction of that on $\big (E,\langle\cdot,\cdot\rangle\big)$, the definition of the Riemannian Hessian yields
    \begin{equation*}
        \Hess f(p)[v] = \nabla_v \nabla f = \nabla_v \grad \overline{f} = \Hess \overline{f}(p)[v],
    \end{equation*}
    for all $p\in M$ and $v \in T_pM$.
\end{proof}
\begin{remark}
    Observe that, since the Riemannian gradient $\nabla f$ is continuous in this setting, so is the Riemannian Hessian $\Hess f(p)$, owing to the continuity of the Levi-Civita connection. 
\end{remark}
These results transfer to open subsets of weak Riemannian manifolds, modulo the respective continuity arguments for the Riemannian gradient and Hessian. 
\begin{lemma}
    Let $(M,g)$ be a weak Riemannian $C^1$-manifold and $U\subset M$ be an open subset. Restricting the metric $g$ to $U$ yields a weak Riemannian manifold $(U,g)$. Let $f\colon U\to \mathbb{R}$ be $C^1$ with a $C^1$-extension $\overline{f}\colon M\to \R$, such that $\overline{f}$ is a gaf. Then the Riemannian gradient on $U$ coincides with that of the extension:
    \begin{equation*}
        \nabla f(p) = \nabla \overline{f}(p),\quad \forall p\in U.
    \end{equation*}
    Moreover, if $(M,g)$ is a Hesse manifold, so is $(U,g)$, and 
    \begin{equation*}
        \Hess f(p) = \Hess \overline{f} (p), \quad \forall p \in U.
    \end{equation*}
\end{lemma}
In the following, we present two illustrative examples of weak Riemannian manifolds. For each example, we derive the corresponding Riemannian gradient, and for the second example, we additionally compute the Riemannian Hessian.
\begin{example}
\label{example: Riemannian Gradient on Imm}
 We revisit the manifold $\text{Imm}(\SSS^1, \R^2)$ discussed in Example \ref{ex:5.10}, however we endow it this time with the weak Riemannian manifold structure of the invariant $L^2$-metric (cf. Example \ref{ex:Sobolev} and Example \ref{ex:basic})
\[g_{inv ,c} (u,w)=\int_{\SSS^1} \langle u,w\rangle \lVert\dot{c}\rVert\, d\mu \qquad c\in \text{Imm}(\SSS^1,\R^2),\]
where we used the identification $T_c \text{Imm} (\SSS^1,\R^2)\cong C^\infty (\SSS^1,\R^2)$ and the inner product is the Euclidean inner product of $\R^2$.
We consider again the length function from Example \ref{ex:5.10}
\[\mathcal{L} \colon \text{Imm}(\SSS^1,\R^2) \rightarrow \R, \mathcal{L}(c):=\int_{\SSS^1} |\dot{c}|\, d\mu.\]
As in \cite{SWW23}, an easy computation shows that the derivative of the length function is 
\begin{align}\label{eq:diff_length_2}
d\mathcal{L}(c;u) = \int_{\SSS^1} -k_c \langle N_c, u\rangle |\dot{c}|\, d_\mu = g_{inv, c} (-k_c N_c, u),\end{align}
where $N_c (z)= (-y_z (z),x_z(z))^\top$ is the normal vector to the curve $c(z)=(x(z),y(z))$ and $k_c$ is the signed curvature scalar at $c$. Thus 
\begin{equation*}
    \nabla \mathcal{L}(c) = -k_cN_c \in C^\infty (\SSS^2,\R^2).
\end{equation*}
\end{example}

The following example showcases a classical application of the Hessian of an energy function which was originally considered to study geodesic loops in Riemannian manifolds, see e.g. \cite{Kli95}.

\begin{example}\label{ex:Eli_curves}
Let $M$ be a strong Riemannian manifold and denote by $H^1(\mathbb{S}^1,M)$ the space of all Sobolev $H^1$-loops with values in $M$, cf. \cite[Section 2.3 and 2.4]{Kli95} for the construction and more information on these manifolds.
In \cite{Eli74} the energy function 
\[E \colon H^1 (\mathbb{S},M) \rightarrow \R , E(x)=\frac{1}{2}\int_0^1 \lVert \partial x (s)\rVert^2 \, ds=\frac{1}{2}\lVert \partial x \rVert_{L^2}^2\]
is defined, where $\partial x$ is the $L^2$-tangent field induced by the loop $x$. The energy function is of interest as its critical points are geodesics. 
The gradient of $E$ with respect to the Sobolev $H^1$-metric (constructed as in Example \ref{ex:Sobolev} but with respect to the (noninvariant) $L^2$-metric from Example \ref{ex:basic}) are computed in \cite{Eli74} as follows:
\begin{align*}
\nabla E (x) &= -(1-\Delta)^{-1} \nabla \partial x
\end{align*}
where the $\nabla$ on the right is the covariant derivative induced by the metric on $M$, $\Delta$ is the Laplace-Beltrami Operator (mapping $H^1$-loops to $H^{-1}$-loops) and one exploits that $(1-\Delta)$ is a compact invertible operator. Then the Hessian at $\xi_x \in T_x H^1 (\mathbb{S},M)$ is given by
\begin{align*}
\Hess E (\xi_x) &= \xi_x  +(1-\Delta)^{-1}(R(\partial x, \xi_x)(1-\Delta)^{-1}(\partial x)-\nabla\left(R(\partial x, \xi_x)\nabla E(x)\right)-\xi_x)
\end{align*}
where $R$ is the curvature tensor of $M$.
As remarked in \cite[p. 114]{Eli74}, the Hessian is the identity plus a compact operator and at a critical point, the null-space of the Hessian consists of all closed Jacobi fields along the critical point (which is an $M$-valued loop!). Note that the tangent field $\partial x$ is such a critical point and this corresponds to the fact that there is a whole circle of critical points in $H^1(\mathbb{S},M)$ obtained by rotating the geodesic $x$.

While the structure of critical points is more complicated than in the finite dimensional matrix case (critical points piling up), the Hessian can nevertheless be used to study convergence of gradients towards the critical point, see e.g. \cite[Theorem B]{Eli74}.
\end{example}
\section{Numerical Experiments}
\label{sec: num exp}
In this chapter, we apply the developed optimization methods to specific examples. Note that the numerical methods employed here are using symbolic computations to avoid discretizations of the infinite-dimensional manifolds. As pointed out, we are not aware of results on discretizations of infinite-dimensional (weak) Riemannian manifolds for numerical or optimization purposes. Hence we postpone studying discretizations for later work. Employing first- and second-order optimality conditions, we locate critical points, ascertain their nature as extrema where applicable, and implement RGD. The examples satisfy all Assumptions of Proposition \ref{proposition: RGD gradient convergence} and therefore exhibit the anticipated convergence of $\vvvert \nabla f(c_k)\vvvert_{c_k}$ to zero and of the iterates to a minimizer.

%In contrast, the final example demonstrates the breakdown of these guarantees when any assumption fails. All computations are performed in the pre-shape space framework, i.e. on open subsets of $C^\infty(\SSS^1,\R^2)$ with a given weak Riemannian metric. 
\begin{example}
\label{ex:emb1}
We consider the locally convex space $C^{\infty}(\mathbb{S}^1,\mathbb{R}^2)$ endowed with the $L^2-$ metric $g(h,k) =\int_{\mathbb{S}^1}\langle h(\theta),k(\theta)\rangle d\theta$. Since $\text{Emb}(\mathbb{S}^1,\mathbb{R}^2)$ is an open subset of $C^{\infty}(\mathbb{S}^1,\mathbb{R}^2)$, the pair $\big(\text{Emb}(\mathbb{S}^1,\mathbb{R}^2),g\big)$ constitutes a weak Riemannian manifold.
We aim to minimize
\begin{equation*}
        f\colon \text{Emb}(\mathbb{S}^1,\mathbb{R}^2) \to \mathbb{R}, \quad c \mapsto \int_{\mathbb{S}^1}\|c(\theta)-\theta\|^2d\theta.
\end{equation*}
using the Riemannian gradient descent as introduced in Algorithm 1.

The function $f$ admits a smooth extension on $C^\infty(\SSS^1,\R^2)$ given by the same expression. A direct computation shows that the gradient of this extension is given pointwise by $\grad\overline{f}(c)(\theta) = 2(c(\theta)-\theta)$. By the extension result of Riemannian gradients ~\ref{lemma: Extension Riemannian Gradient}, the Riemannian gradient of $f$ on $\text{Emb}(\SSS^1,\R^2)$ is therefore $\nabla f(c) = 2(c-\text{Id}_{\mathbb{S}^1})$.
Consequently, a point $c\in \text{Emb}(\SSS^1,\R^2)$ is a critical point of $f$ if and only if $c = \text{id}_{\SSS^1}$. Since $f(c) \geq 0$ for all $c \in \text{Emb}(\SSS^1,\R^2)$ and $f(\text{id}_{\SSS^1})=0$, the identity embedding is the unique global minimizer of $f$.
To apply the Riemannian gradient descent, consider step sizes $\alpha_k > 0$ for $k\in \mathbb{N}$. Since the weak Riemannian manifold under consideration is an open subset of a locally convex space, the tangent space at any point $c$ is isomorphic to the space $C^\infty(\SSS^1,\R^2)$ itself. Therefore, we assume that for sufficiently small step sizes the iterates remain within this open subset, and consequently no retraction needs to be defined. For the resulting sequence of iterates $(c_k)_{k\in \mathbb{N}}$, a direct computation shows that 
\begin{equation*}
    f(c_k)-f(c_{k+1}) = \alpha_k(1-\alpha_k)\vvvert\nabla f(c_k)\vvvert_{c_k}^2,\quad \forall k\in \mathbb{N}.
\end{equation*}
Hence, if there exists a constant $\rho >0$ such that the step-sizes $\alpha_k$ satisfy $\rho \leq \alpha_k(1-\alpha_k)$ for all $k \in \mathbb{N}$, the sufficient decrease condition stated in Assumption \ref{assumption: iteration decrease} is fulfilled. In particular, for a constant step-size $0<\alpha<1$, this is satisfied for $\rho = \alpha(1-\alpha)$.

Since $f$ attains a global minimum and $\nabla f$ is sequentially continuous, all assumption of the general convergence result ~\ref{proposition: RGD gradient convergence} are fulfilled. Consequently, every accumulation point of the sequence of iterates $(c_k)_{k\in \mathbb{N}}$ is a critical point of $f$ and the gradient norms $\vvvert \nabla f(c_k)\vvvert_{c_k}$ converge to zero. Moreover, for every $K\geq 1$, there exists $k \in \{0,...,K-1\}$ such that 
\begin{equation*}
        \vvvert\nabla f(c_k)\vvvert_{c_k}\leq \sqrt{\frac{f(c_0)}{\rho}}\frac{1}{\sqrt{K}}.
\end{equation*}
We conclude with a numerical illustration of the above convergence behavior: \begin{figure}[H]
    \centering
    \includegraphics[width=0.9\linewidth]{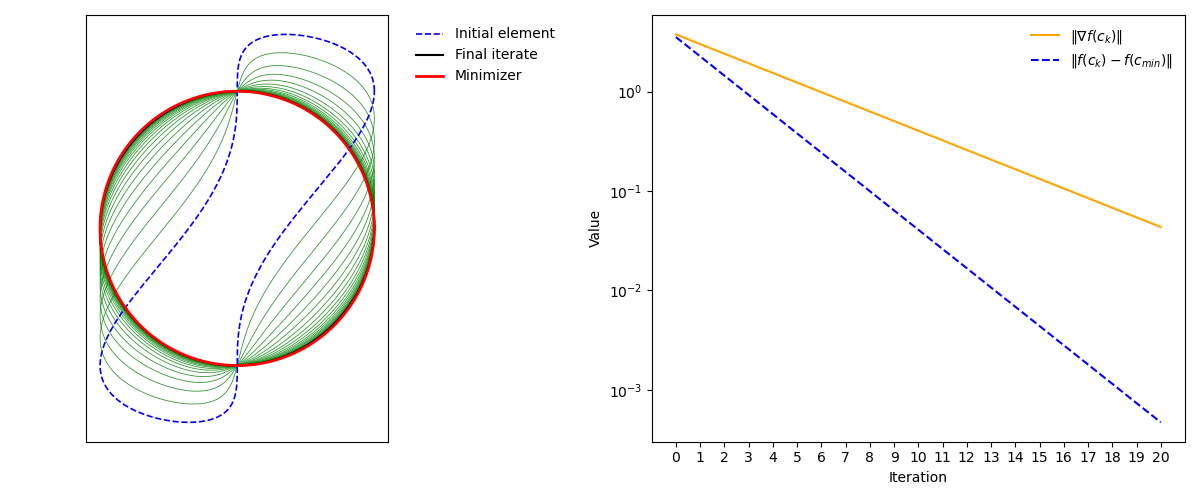}
    \caption{Riemannian gradient descent for $f$. Left: evolution of the iterates. Right: function values and gradient norms over twenty iterations.}
    \label{fig:exp1}
\end{figure}Figure~\ref{fig:exp1} shows twenty iterations of the Riemannian gradient descent with constant step-size $\alpha = 0.1$, starting from the initial embedding $c_0(x,y) = (x^3,x+y)$. The left panel depicts the evolution of the iterates, while the right panel displays the decrease of the function values and the norms of the Riemannian gradients, in agreement with the theoretical convergence results.

\end{example}

\begin{example}
    As in the Example \ref{ex:emb1}, we consider the weak Riemannian manifold $\big(\text{Emb}(\SSS^1,\R^2),g\big)$. Using the Riemannian gradient descent, we now aim to minimize the function
    \begin{equation*}
        f_g\colon \text{Emb}(\SSS^1,\R^2) \to \R, \quad c\mapsto \int_\SSS \|c(\theta)-g(\theta)\|^2d\theta + \lambda \int_\SSS c(\theta)^2d\theta 
    \end{equation*}
    for some $g \in C^\infty(\SSS^1,\R^2)$ and $\lambda\geq 0$.
    
    Proceeding as in the previous example, we obtain the following expression for the Riemannian gradient of $f_g$: 
    \begin{equation*}
        \nabla f_g (c) = 2\big( (1+\lambda)c-g\big) 
    \end{equation*}
    Thus, $f_g$ admits a unique critical point given by 
    \begin{equation*}
        c = \frac{g}{1+\lambda}.
    \end{equation*}
    In order to verify that this critical point is indeed a minimizer of $f_g$, we investigate the Riemannian Hessian. To this end, we first introduce a Levi-Civita connection on $\text{Emb}(\SSS^1,\R^2)$.
    We identify vector fields on $\text{Emb}(\SSS^1,\R^2)$ with mappings 
    \begin{equation*}
        X \colon \text{Emb}(\SSS^1,\R^2) \to C^\infty(\SSS^1,\R^2).
    \end{equation*}
    Following the construction of Schmeding in \cite[5.7]{Sch23}, which is based on the use of connectors, the Levi-Civita connection on $\text{Emb}(\SSS^1,\R^2)$ is defined as follows. 
    \begin{equation*}
        \big(\nabla_h Y\big)(c) = dY(c;h), \quad c\in \text{Emb}(\SSS^1,\R^2), Y \in \mathcal{V}(\text{Emb}(\SSS^1,\R^2)), h \in C^\infty(\SSS^1,\R^2).
    \end{equation*}
    Throughout, we suppress the notation associated with these identifications for simplicity. Consequently, the Riemannian Hessian of $f_g$  at $c\in \text{Emb}(\SSS^1,\R^2)$ is given by 
    \begin{equation*}
        \text{Hess}f(c)[h] = \big(\nabla_u \nabla f\big)(h) = d\nabla f(c;h) = 2(1+\lambda)h, \quad h \in C^\infty(\SSS^1,\R^2)
    \end{equation*}
    Thus, the Riemannian Hessian is positive definite for all $c \in \text{Emb}(\SSS^1,\R^2)$ provided that $\lambda > -1$. Moreover, $\Hess f(p)$ is coercive as $g_c(\Hess f(c)[h],h) = \frac{2}{1+\lambda}\vvvert h\vvvert^2_c$ for all $h\in C^\infty(\SSS^1,\R^2)$. Then, by ~\ref{proposition: Hesse minimizer}, the second-order critical point $c = \frac{g}{1+\lambda}$ is indeed a minimizer of $f_g$. 
    
    To apply the Riemannian gradient descent from Algorithm 1, let $(\alpha_k)_{k\in \mathbb{N}} \subset (0,\infty)$ denote a sequence of step-sizes. For sufficiently small step-sizes, we again assume that the iterates remain within the open set $\text{Emb}(\SSS^1,\R^2)$, which allows us to avoid defining a retraction. For the resulting sequence of iterates $(c_k)_{k\in \mathbb{N}}$, a straightforward computation yields 
    \begin{equation*}
        f(c_k) -f(c_{k+1}) = \alpha_k \big(1-(1+\lambda)\alpha_k\big)\vvvert\nabla f_g(c_k)\vvvert^2_{c_k} \quad \forall k\in \mathbb{N}.
    \end{equation*}
    Hence, the sufficient decrease Assumption \ref{assumption: iteration decrease} is satisfied provided that, for all step-sizes $\alpha_k$ there exists a constant $\rho >0$ such that $\rho \leq \alpha_k\big( 1-(1+\lambda)\alpha_k\big)$. For a constant step-size $0<\alpha< \frac{1}{1+\lambda}$, the choice $\rho = \alpha\big(1-(1+\lambda)\alpha\big)$ satisfies this condition.
    
    As $f_g$ admits a global minimizer and the Riemannian gradient $\nabla f_g$ is sequentially continuous, the decrease of the Riemannian gradient norm stated in \ref{proposition: RGD gradient convergence} follows. Furthermore, all accumulation points of the resulting iterative sequence are critical points and for every $K\geq 1$, there exists an index $k\in \{0,...K-1\}$ such that
    \begin{equation*}
        \vvvert\nabla f_g(c_k)\vvvert_{c_k} \leq \sqrt{\frac{f(c_0)}{\rho}}\frac{1}{\sqrt{K}}.
    \end{equation*}
    Consider the smooth map 
    \begin{equation*}
        g\colon \SSS^1 \to \R^2,\quad (x,y) \mapsto \big(x,\frac{3}{2}y\big)
    \end{equation*}
    and the smooth embedding chosen as the initial iterate, 
    \begin{equation*}
        c_0\colon \SSS^1 \to \R^2,\quad (x,y)\mapsto (x^3,x+y).
    \end{equation*}
    
    Figure~\ref{fig:exp2} illustrates the behavior of the Riemannian gradient descent with constant step-size $\alpha = 0.04$ and parameter $\lambda = 0.7$. The left panel shows the evolution of the iterates $c_k$ under the Riemannian gradient descent. The right panel depicts the decrease of the function value $f_g(c_k) -f_g(c_{min})$ in norm, together with the norm of the Riemannian gradient $\|\nabla f_g(c_k)\|_{c_k}$, over twenty iterations.
 
    \begin{figure}[H]
        \centering
        \includegraphics[width=1.0\linewidth]{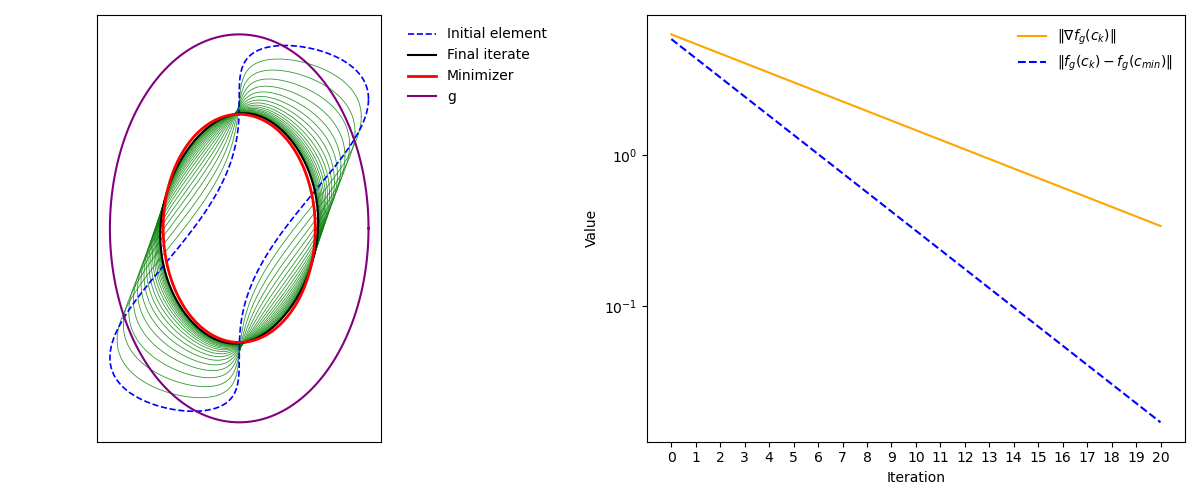}
        \caption{Riemannian gradient descent for $f_g$. Left: evolution of the iterates. Right: function values and gradient norms over twenty iterations.}
        \label{fig:exp2}
    \end{figure}
\end{example}

\appendix

\section{Sprays, connections and metrics}\label{App:diffgeo}

In this section we recall some standard material. For Banach manifolds this can be found e.g. in \cite{Lang01,KaS25}. First we need the following for a tangent bundle $TM$ of a smooth manifold: For every $\lambda \in \mathbb{R}$ we let $h_\lambda \colon TM \rightarrow TM$ be the vector bundle morphism which in every fiber $T_xM$ is given by multiplication with $\lambda$.

\begin{definition}
Let $M$ be a smooth manifold. A \emph{spray} is a vector field $S\in \mathcal{V}(TM)$ on $TM$, i.e. a map $S\colon TM \rightarrow T(TM))$ such that $T\pi_{M} \circ S = \mathrm{id}_{TM}$ and  for all $\lambda \in \R$, we have 
$$S\circ h_\lambda = Dh_\lambda (\lambda S). $$
\end{definition}
In local coordinates $(U,\varphi)$ for $M$, a spray $S \colon TM \rightarrow T^2M$ can be expressed as $S_U(x,v)=(x,v,v,S_{U,2}(x,v))$, where $S_{U,2}(x,\lambda v) = \lambda^2 S_{U,2}(x,v)$. It is easy to see (cf. \cite[4.3]{Sch23}) that in every chart $(U,\varphi)$ to a spray there is an associated quadratic form and a bilinear form given by the formulae 
$$\Gamma_U (x,v):=\frac{1}{2}d_2^2 S_{U,2}(x,0;(v,v))=S_{U,2}(x,v)\qquad B_U(x,v,w)= \frac{1}{2}d_2^2 S_{U,2}(x,0;(v,w)).$$
Sprays provide the vector fields formalizing second order differential equations on manifolds.

\begin{definition}
Let $(M,g)$ be a weak Riemannian manifold. The spray $S$ is called \emph{metric spray} (or \emph{geodesic spray}) if locally in every chart domain $U$ the associated quadratic form $\Gamma_U$ 
satisfies for all $v,w \in T_x U$ the relation
\begin{equation}\label{eq:mspray}
g_U(x,\Gamma_U(x,v),w)=\frac{1}{2}d_1 g_U (x,v,v;w)-d_1g_U(x,v,w;v),
\end{equation}
where we view $g$ locally as a map of three variables and $d_1$ denotes the partial derivative with respect to the first component.
\end{definition}

On a strong Riemannian metric \eqref{eq:mspray} can be used to define the quadratic form $\Gamma_U$. 
Note that the spray is a coordinate base independent way to describe the quadratic object usually described as the metrics Christoffel symbols. There are examples (\cite[Example 4.22]{Sch23}) of weak Riemannian metrics without an associated metric spray. Unsurprisingly, metric sprays are stable under isometric isomorphism. We provide the proof here for the readers convenience as it showcases how sprays transform under diffeomorphisms.

\begin{lemma}
Let $F \colon (M,g) \rightarrow (N,h)$ be a Riemannian isometry between weak Riemannian manifolds. Then $(N,h)$ admits a metric spray if and only if $(M,h)$ admits one.
\end{lemma}

\begin{proof}
The situation is symmetric, whence it suffices to assume that $(N,h)$ admits the metric spray $S^h$. Observe that
$S^g:= T^2(F^{-1}) \circ S^h \circ TF$ is a spray, cf. \cite[Lemma 3.9]{KaS25}.

To check that $S^g$ is a metric spray, one simply has to observe that the relation \eqref{eq:mspray} for the quadratic form of $S$ directly yields the desired relation for the quadratic form of $S^g$ in suitable charts. For the readers convenience we spell this out explicitely: Fix a chart $(U,\varphi)$ of $N$ and obtain the the chart $(F^{-1}(U),\varphi \circ F)$ of $M$. Since $F$ is a diffeomorphism it suffices to compute in charts of this type that $S^g$ is the metric spray. Note that by construction as $S^g=T^2F^{-1}\circ S\circ TF$ the local representative 
\[T^2(\varphi\circ F)\circ S^g \circ T(\varphi\circ F)^{-1} = T^2\varphi^{-1} \circ S \circ T\varphi^{-1}\] of $S^g$ in the $\varphi\circ F$ chart coincides with the local representative of $S$ in the chart $\varphi$. We deduce that the quadratic forms $\Gamma_U$ for $S$ on $\varphi(U)$ and $\Gamma_U^g$ for $S^g$ on $\varphi(U)$ coincide. 

Now pick $x \in \varphi(U), v,w\in T_x\varphi(U)$ and since $F$ is a Riemannian isometry 
\begin{align*}
g_{F^{-1}(U)} (x,v,w)&= g_{(\varphi\circ F)^{-1}(x)} (T_x (\varphi \circ F)^{-1}(v),T_x (\varphi \circ F)^{-1}(w))  \\ &= g_{F^{-1}\varphi^{-1}(x)} (T_x (F^{-1}\varphi^{-1})(v),T_x F^{-1}\varphi^{-1}(w)) \\&=h_{\varphi^{-1}(x)}
(T_x \varphi^{-1}(v),T_x\varphi^{-1}(w)) = h_U(x,v,w).\end{align*} We compute locally in the pair of charts $(U,\varphi)$ and $(F^{-1}(U),\varphi\circ F)$ and since the local representatives of the metrics coincide and \eqref{eq:mspray} holds for $h_U$ and $\Gamma_U$, we deduce from the fact that the quadratic forms coincide that $Q^g_U$ satisfies \eqref{eq:mspray}.
\end{proof}

Every spray induces a covariant derivative (see e.g. \cite[Proposition 4.3.9]{Sch23}). 

\begin{definition}
Let $S\colon TM\rightarrow T(TM)$ be a spray, then there exists a unique covariant derivative $\nabla \colon \mathcal{V}(M) \times \mathcal{V}(M) \rightarrow  \mathcal{V}(M)$ such that in a chart $(\varphi ,U)$, the local formula
\begin{equation}\label{cod:loc_spray}
\nabla_U (u, Y) (x)= dY(x;u(x))-B_U (x,u(x),Y(x))
\end{equation}
holds.
We call $\nabla$ the covariant derivative associated to the spray $S$.
\end{definition}
A covariant derivative on a weak Riemannian manifold $(M,g)$  is called \emph{metric derivative} if it is compatible with $g$ in the sense that
\begin{align}\label{eq:metric_deriv}
X.g(Y,Z)=g(\nabla_X Y, Z) + g(Y,\nabla_X Z), \qquad X,Y,Z \in \mathcal{V}(M),\end{align}
where we use the shorthand $X.f := Df \circ X$. 
Note that a spray is the metric spray for a Riemannian metric if and only if the associated covariant derivative is a metric derivative.

The second order differential equations described by a spray are variants of geodesic equations. As for a Riemannian metric, if one can solve these differential equations, they give rise to an exponential map associated to the spray. We recall from \cite{Lang01}:

\begin{example}
If $M$ is a paracompact Banach manifold with a spray $S\colon TM \rightarrow T(TM)$, then  the spray exponential $\exp_S \colon TM\supseteq \Omega \rightarrow M$ is a normalized local addition on $M$.
\end{example}

\bibliographystyle{abbrv}
\bibliography{opt_wR}

\end{document}